\documentclass{amsart}
\usepackage{graphicx}
 \usepackage{stmaryrd}
\usepackage{color}
\def\a{\alpha}
\def\b{\beta}

\def\({\left (}
\def\){\right )}
\def\<{\left\langle}
\def\>{\right\rangle}

 \newtheorem{thm}{Theorem}[section]

\newtheorem{lem}[thm]{Lemma}

\newtheorem{rem}[thm]{Remark}
\newtheorem{acknowledgement}{Acknowledgement}

\newcommand{\norm}[1]{\left\Vert#1\right\Vert}
\newcommand{\abs}[1]{\left\vert#1\right\vert}

\newcommand{\R}{\mathbb R}

\newcommand{\pfrac}[2]{\frac{\partial #1}{\partial #2}}

\numberwithin{equation}{section}

\begin{document}

\title{ The  rectified $n$-harmonic map flow with  applications to  homotopy classes}

\author {Min-Chun Hong }

\address{Min-Chun Hong, Department of Mathematics, The University of Queensland\\ Brisbane, QLD 4072, Australia}  \email{hong@maths.uq.edu.au}

\begin{abstract}
We introduce a  rectified $n$-harmonic map flow from an $n$-dimensional closed Riemannian manifold to another closed  Riemannian manifold. We  prove existence of a global
solution, which is regular except for a finite number of points,  of the rectified $n$-harmonic map flow   and establish an energy identity for the flow at each singular
time.   Finally, we present two applications of the rectified $n$-harmonic map flow to minimizing the $n$-energy functional and the Dirichlet energy functional in a homotopy
class.
\end{abstract}

\maketitle

\markright {The  rectified $n$-harmonic map flow}

\section{Introduction}

Let $(M,g)$ be an $n$-dimensional compact Riemannian manifold without
boundary, and let $(N,h)$ be another
$m$-dimensional compact Riemannian manifold without boundary (isometrically embedded into $\R^L$).
 The $n$-energy functional $E_n(u; M)$ of a map $u: (M,g)\to (N,h)$  is defined by
\[E_n(u; M)=\frac 1 n\int_M  |\nabla u |^n\,dv.\]
A map $u$ from $M$ to $N$ is said to be an  $n$-harmonic map  if
$u$ is a critical point of the $n$-energy functional; i.e.   it satisfies
\begin{equation} \mbox{div}
\left[  |\nabla
u|^{n-2} \nabla u\right]
+|\nabla
u|^{n-2}A(u)(\nabla u,\nabla u)=0\quad  \mbox { in } M,\end{equation}
where $A$ is the second fundamental form of $N$.

 When $n=2$, an  $n$-harmonic map is a harmonic map. The fundamental  question on harmonic maps, asked by Eells and Sampson \cite{ES} (see also \cite{EL}), is  whether  a
given smooth map $u_0$ can be deformed to a harmonic map  in its homotopy class $[u_0]$. Eells and Sampson \cite{ES} answered the question  for the case that the sectional
curvature of  $N$ is non-positive   by introducing the heat flow for harmonic maps. In order to solve the   Eells-Sampson question,  it is  very important to establish  global
existence of the harmonic map flow.
When $n=2$, Struwe \cite{St} proved  global existence of the weak solution  to
the harmonic map flow, where the solution is smooth except for a finite set of singularities.   Chang, Ding and Ye \cite{CDY}
constructed a  counter-example that the harmonic map flow  blows up at finite time.  Ding and Tian \cite {DT} established the energy identity of the harmonic map flow at
each blow-up time through a finite number of  harmonic maps  on $S^{2}$ (called bubbles).  Qing and Tian \cite {QT} proved  that as $t\to \infty$,  there is no neck between a
limit map $u_{\infty}$ and  bubbles. Therefore, a given map $u_0$ can be deformed into a splitting sum of  finite harmonic maps.

  When $n>2$,  Chen and Struwe  \cite {CS} showed  global existence of a weak solution of the harmonic map flow, in which the weak solution   is  partially regular and has a
complicated singular set.  In general,  it is difficult  to apply the harmonic map flow to investigate the Eells-Sampson  question. Motivated by the Eells-Simpson question, it
is interesting to ask whether  a given map $u_0\in C^{\infty}(M, N)$  can be deformed to an $n$-harmonic map in the homotopy class $[u_0]$.
Related to this question,
Hungerbuhler \cite {Hung}   investigated the $n$-harmonic map flow in the following equation:
\begin{equation} \label{n-flow}\frac {\partial u}{\partial t}= \mbox{div} \left[  |\nabla
u|^{n-2} \nabla u\right]
+|\nabla
u|^{n-2}A(u)(\nabla u,\nabla u)\end{equation} with initial value $u_0$, and  generalized the result of Struwe \cite{St} to
 prove that there
exists a global weak solution $u: M\times [0,+\infty)\to N$ of  the $n$-harmonic map flow (\ref{n-flow})  such that $u\in
{\bf C}^{1,\alpha }(M\times (0,+\infty )\backslash
\{\Sigma_k\times T_k \}_{k=1}^L)$ for a finite
number of singular times $\{T_k\}_{k=1}^L$ and a finite number of singular
closed sets $\Sigma_k \subset M$ for $k=1,..., L$  with an integer
$L$, depending only $M$ and $u_0$. Chen,  Cheung,  Choi, Law \cite {CCCL} constructed a counter-example to show that the $n$-harmonic map  flow  (\ref{n-flow})  blows up at finite time for
$n=3$.
However,  it has been an open question  whether the singular set $\Sigma_k$ of the  $n$-harmonic map flow   at each singular time $T_k$ is finite.  Without the finiteness of
the singular set $\Sigma_k$, it
 is difficult to control the loss of the energy at the singular time $T_k$.  In order to overcome this difficulty,
 we introduce a rectified $n$-harmonic map flow  in the following equation:
\begin{equation} \label{NA}
(1-a+a|\nabla u|^{n-2}) \frac {\partial u}{\partial t}= \mbox{div} \left[  |\nabla
u|^{n-2} \nabla u\right]
+|\nabla
u|^{n-2}A(u)(\nabla u,\nabla u)\end{equation}
with initial value $u(0)=u_0$ with a  constant $a\in [0, 1]$.
In particular, when $a=0$, the flow (\ref{NA}) is the standard $n$-harmonic map flow. When $a=1$,  the flow
(\ref{NA}) is an evolution equation involving the normalized   $n$-Laplacian (e.g. \cite {Do}).

In this paper, we firstly prove:

\begin {thm}\label{Theorem 1}For each $a\in (0, 1]$, there
exists a global weak solution $u: M\times [0,+\infty)\to N$ of (\ref{NA}) with initial value $u_0\in W^{1,n}(M)$ in which there are  finite
times $\{T_k\}_{k=1}^L$ and   finite   singular points $\{x^{j,k}\}_{j=1}^{l_k}$   such that
$u$ is regular in $M\times (0,+\infty )\backslash
\{\{x^{j,k}\}_{j=1}^{l_k}\times T_k \}_{k=1}^L$ in the following sense:
\begin{align*}
&u\in
 C_{loc}^{0,\alpha }(M\times (0,+\infty )\backslash
\{ \{x^{1,k} \cdots, x^{l_k,k}\}\times T_k \}_{k=1}^L),\\
& \nabla u\in L_{loc}^{\infty} (M\times (0,+\infty )\backslash
\{  \{x^{1,k} \cdots, x^{l_k,k}\}\times T_k \}_{k=1}^L) .\end{align*}
As $t\to T_k$, $u(x,t)$ strongly converges to $u(x,T_k)$ in $W_{loc}^{1,n+1}(M\backslash \{x^{1,k} \cdots, x^{l_k,k}\})$.
\end{thm}

Theorem \ref{Theorem 1} generalized  the result of Struwe \cite{St}.
For the proof of Theorem \ref{Theorem 1}, one of key ideas  is to obtain an $\varepsilon$-regularity estimate by improving  the delicate proof of Hungerbuhler in  \cite{Hung}
for the case of  $a=0$ based on a variant of Moser's iteration.
Since the term $|\nabla u|^{n-2}\partial_t u$ in the flow (\ref{NA})  causes an extra difficulty,   we have to carry out much  more complicated analysis   to obtain the
boundedness of $|\nabla u|$ (see Lemma   \ref{higher-regularity}).

\begin{rem}
We would like to point out that the rectified $n$-harmonic map flow   is related to an evolution equation involving the normalized   $p$-Laplacian (e.g. \cite {Do}).
It will be very interesting if some can prove that the solution of the flow (\ref {NA}) is $C^{1,\alpha}$. \end{rem}

Secondly, we generalize    the result of Ding-Tian \cite {DT} from two-dimensional case to $n$-dimensional cases and prove:

\begin{thm}\label{Theorem 2} For   each $a\in (0, 1]$, let $u: M\times [0,+\infty)\to N$ be a solution of (\ref{NA}) with initial value $u_0$ in Theorem \ref{Theorem 1}.
Let $T_k$ be the above singular time.
Then, there are a finite number of
$n$-harmonic maps $\{\omega_{i,k}\}_{i=1}^{m_k}$ (also called bubbles) on $S^{n}$ such that
 \[\lim_{t \nearrow T_k}  E_n(u(t); M)= E_n(u(\cdot ,T_k); M)+\sum_{i=1}^{m_k}   E_n  (\omega_{i,k} , S^{n}). \]
\end{thm}
For the proof of  the energy identity, Wang and Wei \cite {WW} proved an energy identity for a sequence of approximate $n$-harmonic maps by reducing multiple bubbles to a single bubble. In
order to make proofs more clear, we give a  detailed procedure of bubble-neck decomposition based on the method of Ding-Tian \cite{DT} and then prove the energy identity.

\medskip

Next, we will present some applications of the related $n$-flow to minimizing the $n$-energy functional in a homotopy class $[u_0]$.
When $n=2$, Lemaire \cite {Lemaire} and Schoen-Yau
\cite {SY} established  existence results of harmonic maps by minimizing the Dirichet energy  in a homotopy
class under the  topological condition $\pi_2(N)=0$. In \cite{SU}, Sacks and Uhlenbeck established
many existence results
 of minimizing harmonic maps in their homotopy classes by introducing
the `Sacks-Uhlenbeck functional'.
Recently, the author and Yin \cite {HY} introduced the Sacks-Uhlenbeck flow on Riemannian surfaces to provide a new proof of the energy identity of a minimizing sequence in a
homotopy class $[u_0]$. A similar approach on the Yang-Mills $\alpha$-flow on $4$-manifolds has been obtained by the author, Tian and Yin \cite {HTY}. Expanding the idea in
\cite {HY} with
 applications of a rectified $n$-flow,  we prove:
\begin{thm}\label{Theorem 4} For a  homotopy class $[u_0]$, let $\{u_{k}\}_{k=1}^{\infty}$ be a minimizing sequence of $E_{n}$ in the homotopy class $[u_0]$ and   $u$ the weak
limit in $W^{1,n}(M,N)$. Then, there is a  finite  set $\Sigma$ of singular points in $M$  so that as $k\to \infty$, $u_k$ converges strongly  to $u$ in $W^{1,n}_{loc}
(M\backslash\Sigma, N)$  and there are a finite number of $n$-harmonic maps
 $\{\omega_i\}_{i=1}^l$ on $S^{n-1}$ such that
\[\lim_{k \to \infty }  E_n(u_k; M)= E_n(u_{\infty}; M)+\sum_{i=1}^l   E_n  (\omega_i , S^{n-1}). \]
 If $\pi_n (N)=0$,  the singular set $\Sigma$ is empty and there is a minimizing map of the $n$-energy functional in  the homotopy class $[u_0]$.
 \end{thm}
 We would like to point out that Duzaar and  Kuwert \cite{DK} studied the decomposition of  a minimizing sequences of the $n$-energy functional in a homotopy class $[u_0]$
with $N=S^l$, which could be used to prove an energy identity for the minimizing sequence.  Our proof is completely different from one in \cite {DK}. By a modification of the
above $n$-harmonic flow,  we follow  the idea of the $\alpha$-flow \cite {HY} to  rectify a new minimizing sequence $\{\tilde u_k\}_{k=1}^{\infty}$,
 having the same weak limit $u$ of the  minimizing sequence $\{u_k\}_{k=1}^{\infty}$ in the same homotopy class.

\medskip

Furthermore, in order to prove the existence of a harmonic map  in  a given homotopy class $[u_0]$, it is a  nature way to minimize the Dirichlet functional in the  homotopy
class. Indeed, there were successful results  for $n=2$, which were mentioned above (\cite {Lemaire}, \cite {SY} and \cite {SU}). In higher dimensions, it is very challenging
to minimize the Dirichlet functional in a homotopy class.
White \cite
{W} showed that if $d$ is  the greatest integer strictly less than $p$, a homotopy equivalence is well defined
for neighboring maps after restriction to the $d$-skeleton of $M$
and  there exists a minimizer of the $p$-energy $ E_p(u; M)=\tfrac 1 p\int_M |\nabla u|^pdv$ with prescribed
$d$-homotopy type. White \cite {W} raised an open problem about the partial regularity of the minimum solution of the $p$-energy with prescribed
$d$-homotopy type. In particular, even for $p=2$, the partial regularity theory of Schoen-Uhlenbeck  \cite{SU}  (also Giaquinta-Giusti  \cite {GG})
on an energy minimizing map $u$  in $W^{1,2}(M,N)$ cannot be applied   since
the Sobolev space $W^{1,2}(M,N)$ cannot be approximated by
smooth maps and  a minimizing
map of the Drichlet in $W^{1,2}(M,N)$  is not  in the homotopy class.

Let $\{u_k\}_{k=1}^{\infty}$ be a minimizing sequence of the $p$-energy $E_p$ in the homotopy class $[u_0]$ for $2\leq p\leq n$ and   let $u\in W^{1,p}(M,N)$
be the weak limit  of the minimizing sequence.  Related to the above White problem, it is a very interesting problem whether the limit map $u$ is a weakly $p$-harmonic map
and partially regular.  Motivated by recent results of \cite {GHY} and \cite {HY1},
we  partially answer  the  question by applying a modified  $n$-flow and prove:

 \begin{thm} \label{Theorem 5} Let $p$ be a number with $2\leq p\leq n$. Assume that $N$ is a homogenous  Riemannian manifold  without
boundary. For a given homotopy class $[u_0]$, let $\{u_i\}_{i=1}^{\infty}$ be a minimizing sequence of the $p$-energy $E_p(u; M)$ in the homotopy class $[u_0]$.  Then, there
is a subsequence of $\{u_i\}_{i=1}^{\infty}$ such that
 $u_i$ weakly converges  to  a weak $p$-harmonic map $u$.  Moreover, $u$ belongs to $C^{1,p} (M\backslash\Sigma, N)$ for a closed singular set  $\Sigma\subset M$ and $\mathcal
H^{n-p}(\Sigma)<\infty$, where $\mathcal H^{n-p}$ denotes the Hausdorff measure.
 \end{thm}

 For proving Theorem \ref{Theorem 5}, we employ a perturbation of the $p$-energy functional   and its gradient flow in a homotopy class. This kind  perturbation of the
Drichlet functional was used by Uhlenbeck in \cite {U1} to reprove Eells-Sampson's result, and  was employed by Giaquinta, the author and Yin \cite {GHY} for proving partial
regularity of  the relaxed functional of harmonic maps and also by the author and Yin \cite {HY} for proving partial regularity of the relaxed functional of bi-harmonic maps.

The paper is organised as follows. In Section 2, we  establish some basic estimates  and global existence of  weak solutions to  the rectified $n$-flow.   In Section 3, we
prove the energy identity at a singular time and finish a proof of Theorem  \ref{Theorem 2}.
In Section 4, we prove Theorem \ref{Theorem 4}.
In Section 5, we finish a proof of Theorem \ref{Theorem 5}.

\section{Some  estimates and global existence}

In local coordinates,  the Riemannian metric $g$ on $M$  can be
represented by
\[g = g_{ij}  dx^i\otimes dx^j \] with   a positive
definitive symmetric $n\times n$ matrix $(g_{ij})$.
 The volume element $dv$ of $(M; g)$ is defined by
 \[dv= \sqrt {|g|} dx \quad \mbox{with } |g|=\mbox{det }(g_{ij}).\]
 Note that $(N,h)$ is a
$m$-dimensional compact Riemannian manifold without boundary, isometrically embedded into $\R^L$.
For a map $u:M\to N$,  the  gradient norm $|\nabla u| $  is given by
\[  |\nabla u (x)|^2 = \sum_{i,j, \alpha}g^{ij}(x) \frac {\partial u^{\alpha}}{\partial x_i}\frac {\partial u^{\alpha}}{\partial x_j},\]
where $(g^{ij})=(g_{ij})^{-1}$ is the inverse matrix of $(g_{ij})$. A $C^{1,\alpha}$-map $u$ from $M$ to $N$ is called an  $n$-harmonic map if it satisfies
\begin{equation}\label{har} \frac{1}{\sqrt{|g|}}\frac{\partial}{\partial x_{i}}\left[  |\nabla
u|^{n-2}g^{ij}\sqrt{|g|}\frac{\partial}{\partial x_{j}}u\right]
+|\nabla
u|^{n-2}A(u)(\nabla u,\nabla u)=0\quad  \mbox { in } M,\end{equation}
where $A$ is the second fundamental form of $N$.

In order to show  existence of the rectified $n$-flow (\ref{NA}),   we consider an approximate $n$-functional
\begin{equation}\label{AP}
    E_{n,\varepsilon}(u)=\int_{M}( \frac  \varepsilon 2  \abs{\nabla u}^2 +\frac 1 {n}|\nabla u|^{n} )\,dv
\end{equation}
for a  constant $\varepsilon>0$.
The Euler-Lagrange equation for the functional (\ref{AP}) is
\begin{equation}\label{AEL} \frac{1}{\sqrt{|g|}}\frac{\partial}{\partial x_{i}}\left[ (\varepsilon + |\nabla
u|^{n-2})g^{ij}\sqrt{|g|}\frac{\partial}{\partial x_{j}}u\right]
+(\varepsilon +|\nabla
u|^{n-2})A(u)(\nabla u,\nabla u)=0.\end{equation}
The rectified gradient flow for  the functional (\ref{AP}) is
 \begin{eqnarray}\label{APH}
 &&(1-a+\varepsilon+a|\nabla u|^{n-2}) \frac {\partial u}{\partial t}\\
 &=&\frac{1}{\sqrt{|g|}}\frac{\partial}{\partial x_{i}}\left[ (\varepsilon + |\nabla
u|^{n-2})g^{ij}\sqrt{|g|}\frac{\partial}{\partial x_{j}}u\right] +(\varepsilon +|\nabla
u|^{n-2})A(u)(\nabla u,\nabla u)\nonumber
    \end{eqnarray}
    with initial value $u(0)=u_0$ in $M$. Multiplying $\partial_tu$ to both sides of (\ref {APH}), we have the following energy identity:
\begin{lem}  \label{EI} Let $u(t)$ be a smooth solution to the  flow (\ref{APH}) in $M\times
 [0,T)$ with initial value $u(0)=u_0$.  Then for each $s$ with  $0<s<T$, we have
\begin{eqnarray} \label{Energy identity}
  &&\int_M\frac {\varepsilon} 2|\nabla u (s)|^2+\frac 1 n|\nabla u(s)|^{n} \,dv +
  \int_0^s \int_M  (1-a+\varepsilon+a |\nabla u|^{n-2} )\left |\frac{\partial u}{\partial t}\right |^2 dv\,dt\\&=& \int_M\frac {\varepsilon} 2 | \nabla u_0|^2+\frac 1 n|
\nabla u_0|^{n}
  \,dv.\nonumber
\end{eqnarray}
\end{lem}
Moreover, we have the following local energy's inequality:
\begin{lem}  \label{Local} (local energy inequality) Let $u(t)$ be a smooth solution to the  flow (\ref{APH}) in $M\times
 [0,T]$ with initial value $u(0)=u_0$ and set $e_{\varepsilon}(u)=\frac {\varepsilon} 2 |\nabla u|^2  +\frac {1}{n} |\nabla u|^{n}$. For any $x_0$ with $B_{2R}(x_0)\subset M$
and  for any two  $s,\tau \in [0, T)$ with $s<\tau$, we have
\begin{align} \label{Local1}& \int_{B_{R}(x_0) }e_{\varepsilon}(u)(\cdot ,\tau)\,dV
+\int_{s}^{\tau }\int_M (1-a+\varepsilon+a|\nabla u|^{n-2})|\partial_t u|^2\,dv\,dt\\
&\leq \,\int_{B_{2R}(x_0)} e_{\varepsilon}(u) (\cdot ,s)\,dv
+  \frac {C(\tau-s)}{R^2}\,\int_M e_{\varepsilon}(u_0)\,dv \nonumber \end{align}
and for each  $a\in (0,1]$, there is a constant $C(a)$ depending on $a$ such that
\begin{align}  \label{Local2}& \quad\int_{B_{R}(x_0) } e_{\varepsilon}(u) (\cdot ,s)\,dv-\int_{B_{2R}(x_0)}  e_{\varepsilon}(u) (\cdot , \tau )\,dv\\
&\nonumber \leq C(a)\int_{s}^{\tau }\int_M (1-a+\varepsilon+a|\nabla u|^{n-2})|\partial_t u|^2\,dv\,dt\\ &+C(a)\left (\frac {(\tau -s)}{R^2}\,\int_M e_{\varepsilon}(u_0)\,dv\,
\int_{s}^{\tau
}\int_M (1-a+\varepsilon+|\nabla u|^{n-2})|\partial_t u|^2\,dv\,dt \right )^{1/2}\,. \nonumber\end{align}
\end{lem}
\begin{proof} Let $\varphi$ be a cut-off function with support in $B_{2R}(x_0)$ and $\varphi\equiv 1$ on $B_{R}(x_0)$ with $|\nabla \varphi|\leq C/R$. Then
\begin{eqnarray*}
    \frac{d}{dt}\int_M \varphi^2e_{\varepsilon}(u) dv
    &=& \int_M \varphi^2  \left <(\varepsilon +|\nabla u|^{n-2})\nabla u, \nabla \frac {\partial u}{\partial t} \right > \,dv\\
    &=&  -  \int_M \varphi^2 (1-a+\varepsilon+a|\nabla u|^{n-2}) \abs{\pfrac{u}{t}}^2 \\
    && +\int_M \varphi  (\varepsilon+|\nabla u|^{n-2})  \nabla u\# \nabla \varphi \# \pfrac{u}{t} \, dv.
\end{eqnarray*}
(\ref {Local1})  follows from integrating in $t$ over $[s,\tau ]$ and using Young's inequality. Similarly, we have (\ref{Local2}).
 \end{proof}

\begin{lem}\label{Lemma 2.3}
    Let $u:M\to N$ be a smooth solution to the  flow equation (\ref {APH}) in $M\times
 [0,T]$ . Set $e_{\varepsilon}(u)=\frac {\varepsilon} 2 |\nabla u|^2  +\frac {1}{n} |\nabla u|^{n}$.
   There is a small constant $\varepsilon_0>0$ such that if  the inequality
\[\sup_{0\leq t\leq T}\int_{B_{2R_0}(x_0)}|\nabla u|^{n} \,dv\,dt
   < \varepsilon_0\] holds for some positive $R_0$,
then we have
   \begin{eqnarray}
  &&\int_{0}^{T}\int_{B_{R_0}(x_0)} |\nabla u|^{n+2}+|\nabla^2 u|^2(\varepsilon +|\nabla u|^{n-2})\,dv\,dt\\
  &&
  \leq  C(1+TR_0^{-2})E_{\varepsilon}(u_0),\nonumber
    \end{eqnarray}
    where the constant $C$ does not depend on $\varepsilon$, $a$ and $u$.
\end{lem}
\begin{proof}
In a neighborhood of each point $x_0\in M$,  we can choose an
orthonormal frame $\{e_i\}_i^n$.  We denote by $\nabla_i$   the
first covariant derivative with respect to $e_i$ and  by $\nabla^2_{ji} u$
the second covariant derivatives of $u$ and so on.

Let $\phi$ be a cut-off function with support in $B_{2R_0}(x_0)$ such that $\phi =1$ in $B_{R_0}(x_0)$, $|\nabla \phi|\leq CR_0^{-1}$ and $|\phi|\leq 1$ in $B_{2R_0}(x_0)$.
Multiplying (\ref {APH}) by $\phi^n\Delta u$, we have
\begin{eqnarray}\label{14}
        &&\int_{B_{2R_0}(x_0)}(1-a+\varepsilon+a|\nabla u|^{n-2}) \left <\partial_t u, \Delta u \right >\phi^n \,dv\\
        &&=\int_{B_{2R_0}(x_0)}\left <\nabla_k( (\varepsilon +|\nabla u|^{n-2})\nabla_k u),  \Delta u\right >\phi^n\,dv\nonumber\\
        &&+\int_{B_{2R_0}(x_0)}\left <(\varepsilon +|\nabla u|^{n-2})A(u)(\nabla u,\nabla u),  \Delta u\right >\phi^n\,dv\nonumber
\end{eqnarray}

In order to estimate the first term of the right-hand side of (\ref{14}),  it follows from the well-known Ricci identity that
 \begin{eqnarray*}\nabla_k \nabla_l \left (  (\varepsilon +|\nabla u|^{n-2})
\nabla u\right )&=&\nabla_l\nabla_k \left (   (\varepsilon +|\nabla u|^{n-2}) \nabla
u\right )\\ &+&R_M\#\left (  ( \varepsilon +|\nabla u|^{n-2}) \nabla u\right )
 \end{eqnarray*}
 with the Riemannian curvature $R_M$.

 Then, integrations by parts twice yield that
\begin{eqnarray*}
      &&\int_{B_{2R_0}(x_0)}\left <\nabla_k( ( \varepsilon +|\nabla u|^{n-2})\nabla_k u), \Delta u\right >\phi^n\,dv\\
        &= & \int_{B_{2R_0}(x_0)}\left <\nabla_l( ( \varepsilon +|\nabla u|^{n-2})\nabla_k u),   \nabla_{k}\nabla_l u\right >\phi^n\,dv\\
        &&- \int_{B_{2R_0}(x_0)}\left <  \nabla_k(|\nabla u|^{n-2}\nabla_k u), \nabla_l u  \right >\nabla_l \phi^n\,dv\nonumber\\
        && + \int_{B_{2R_0}(x_0)}\left <  \nabla_l( ( \varepsilon +|\nabla u|^{n-2})\nabla_k u), \nabla_k u  \right >\nabla_k \phi^n\,dv\nonumber\\
         &&+ \int_{B_{2R_0}(x_0)}\left <R_M\# ( ( \varepsilon +|\nabla u|^{n-2})\nabla_k u),  \nabla_l u  \right >\phi^n\,dv\nonumber \\
        &\geq &   \frac 3 4\int_{B_{2R_0}(x_0)} ( \varepsilon +|\nabla u|^{n-2})|\nabla^2 u|^2\phi^n\,dv \nonumber
         \\
         &&+\frac {n-2}{2}\int_{B_{2R_0}(x_0)} ( \varepsilon +|\nabla u|^{n-2})|\nabla |\nabla u||^2\phi^n\,dv  \\
        &&-C\int_{B_{2R_0}(x_0)} ( \varepsilon +|\nabla u|^{n-2})|\nabla u|^{2}\phi^{n-2} (\phi^2  +|\nabla \phi |^2)\,dv\nonumber .
  \end{eqnarray*}
In order to estimate the  term of the left-hand side of (\ref{14}),  it follows from   integrating by parts and using Young's inequality  that
\begin{eqnarray*}
       &&\quad \int_{B_{2R_0}(x_0)}\left <(1-a+\varepsilon+a|\nabla u|^{n-2}) \partial_t u,\Delta u \right >\phi^n \,dv\\
       &=&-\frac d{dt}\int_{B_{2R_0}(x_0)} (\frac {1-a+\varepsilon}2 |\nabla u|^2+\frac a n|\nabla u|^{n})\phi^n\,dv\\
        &&-a\int_{B_{2R_0}(x_0)}\nabla_k(|\nabla u|^{n-2}) \partial_t u \cdot\nabla_k u  \phi^n\,dv \\
       &&-n\int_{B_{2R_0}(x_0)}(1-a+\varepsilon+a|\nabla u|^{n-2}) \partial_t u \cdot\nabla_k u \phi^{n-1} \nabla_k \phi\,dv\\
       &\leq& -\frac d{dt}\int_{B_{2R_0}(x_0)} (\frac  {1-a+\varepsilon} 2 |\nabla u|^2+\frac a n|\nabla u|^{n})\phi^n\,dv\\
       &&+ C \int_{B_{2R_0}(x_0)}(1-a+\varepsilon+a|\nabla u|^{n-2}) |\partial_t u|^2\phi^n\,dv
       \\
       &&+ \frac {n-2}{4}  \int_{B_{2R_0}(x_0)}|\nabla u|^{n-2} |\nabla(|\nabla u|)|^2 \phi^n\,dv \\
       && + C\int_{B_{2R_0}(x_0)}(1-a+\varepsilon  +a|\nabla u|^{n-2}) |\nabla u|^2|\nabla \phi|^2 \phi^{n-2}\,dv .
  \end{eqnarray*}
Combining above inequalities, we obtain
 \begin{eqnarray}\label{15}
        && \frac d{dt}\int_{B_{2R_0}(x_0)} (\frac {1-a+\varepsilon} 2 |\nabla u|^2+\frac a n|\nabla u|^{n})\phi^n\,dv\\
         &&+\frac 1 2\int_{B_{2R_0}(x_0)}  |\nabla^2 u|^2(\varepsilon+ |\nabla u|^{n-2})\phi^n \,dv\nonumber\\
        &\leq &C\int_{B_{2R_0}(x_0)}  |\nabla u|^2(\varepsilon|\nabla u|^2 + |\nabla u|^{n})\phi^n\,dv\nonumber\\&&
        +C\int_{B_{2R_0}(x_0)} (1+|\nabla u|^{n})\phi^{n-2} (\phi^2  +|\nabla \phi |^2)\,dv
          \nonumber\\
        &&  + C \int_{B_{2R_0}(x_0)}(1+a|\nabla u|^{n-2}) |\partial_t u|^2\phi^n\,dv.\nonumber
  \end{eqnarray}
By applying the H\"older and Sobolev inequalities, we have
  \begin{eqnarray*}
        &&\int_{0}^{T}\int_{B_{2R_0}(x_0)} |\nabla u|^{n+2}\phi^n\,dv\,dt\\
        &&\leq \left (\sup_{0\leq t\leq T}\int_{B_{2R_0}(x_0)} |\nabla u|^{n}\,dv\right )^{\frac 2 n}\int_{0}^{T}\left (\int_{B_{2R_0}(x_0)} |\nabla u|^{\frac
{n^2}{n-2}}\phi^{\frac {n^2}{n-2}}\,dv\right )^{\frac{n-2}n}\,dt\\
        &&\leq C \varepsilon_0^{\frac 2 n} \int_{0}^{T}  \int_{B_{2R_0}(x_0)}|\nabla ( |\nabla u|^{n/2}\phi^{n/2})|^2\,dv \,dt\\
        &&\leq C \varepsilon_0^{\frac 2 n} \int_{0}^{T}  \int_{B_{2R_0}(x_0)}(|\nabla^2  u|^{2}|\nabla u|^{n-2}\phi^{n} +\frac 1 {R_0^2}|\nabla u|^n  )\,dv \,dt.
  \end{eqnarray*}

Integrating (\ref{15}) in $t$ over $[0, T]$, choosing $\varepsilon_0$ sufficiently small and   Lemma \ref{EI}, we have
\begin{eqnarray*}
  &&\int_{0}^{T}\int_{B_{R_0}(x_0)} |\nabla u|^{n+2}+|\nabla^2 u|^2(\varepsilon +|\nabla u|^{n-2})\,dv\,dt \\
  &&\leq C \int_{B_{2R_0}(x_0)}(\frac  {1-a+\varepsilon}2 |\nabla u|^2+\frac a n|\nabla u|^{n})(x,0)\,dv\\
  && \quad + C  (1+ \frac {1}{R_0^2})\int_{0}^{T}E_{\varepsilon}(u;B_{2R_0}(x_0))\,dt\\
  &&\leq C  (1+T+\frac T{R^2_0})E_{\varepsilon}(u_0).
    \end{eqnarray*}
    This proves the claim.
 \end{proof}

For $R>0$ and $z_0=(x_0,t_0)\in M\times (0,\infty)$, we denote
\[P_{R}(z_0)=\{z=(x,t): |x-x_0|<R, t_0-R^2< t\leq t_0\}.\]

\begin{lem} \label{higher-regularity} Let $u$ be    a smooth solution  to the  flow equation (\ref{APH}) with smooth initial value $u_0$.  For any $\beta \geq 1$, there exists
a positive constant
    $\varepsilon_1$ depending on $\beta$  such that  if for some $R_0$ with $0<R_0<\min
\{\varepsilon_1, \frac {t_0^{1/2}}2\}$  the inequality
\[\sup_{t_0-4R_0^2\leq t\leq t_0}\int_{B_{2R_0}(x_0)}|\nabla u|^{n} \,dv
    <\varepsilon_1\] holds,
 then we have
   \begin{eqnarray}
  &&\int_{t_0-R_0^2}^{t_0}\int_{B_{R_0}(x_0)} |\nabla u|^{n+2+\beta }+|\nabla^2 u|^2(\varepsilon +|\nabla u|^{n-2+\beta})\,dv\,dt\\
  &&
  \leq C R_0^n+ C\int_{t_0-4R_0^2}^{t_0}\int_{B_{2R_0}(x_0)}(1+R_0^{-2}) |\nabla u|^{n+\beta }\,dv\,dt, \nonumber
    \end{eqnarray}
 where the constant $C$ does not depend on $\varepsilon$, $u$ and  $a$.

\end{lem}
\begin{proof}In a neighborhood of each point $x_0\in M$,  we still denote by $\nabla_i$   the
first covariant derivative with respect to $e_i$ and  by $\nabla^2_{ij} u$
the second covariant derivatives of $u$ and so on.

 Let $\phi=\phi (x,t)$ be a cut-off function with support in  $B_{R_0}(x_0)\times [t_0-4R_0^2, t_0+4R_0^2]$ such that $\phi =1$ in $B_{R_0}(x_0)\times [t_0-R_0^2, t_0]$,
$|\nabla \phi|\leq C/R_0$, $|\partial_t \phi|\leq \frac 1 {R_0^2}$ and $|\phi|\leq 1$ in  $B_{R_0}(x_0)\times [t_0-4R_0^2, t_0]$.

 Multiplying (\ref {APH}) by $\phi^n|\nabla u|^{\beta}\partial_t u$ and integrating by parts, we have
 \begin{eqnarray*}
        && \int_{P_{2R_0}(x_0, t_0)}( {1-a+\varepsilon}+a |\nabla u|^{n-2})|\nabla u|^{\beta} |\partial_t u|^2\phi^n \,dv \,dt  \\
        &=&-\int_{P_{2R_0}(x_0, t_0)}(\left <(\varepsilon +|\nabla u|^{n-2}) \nabla_k u, |\nabla u|^{\beta}\nabla_k (\partial_t u) \right >\phi^n \,dv \,dt  \\
        &&-\int_{P_{2R_0}(x_0, t_0)}(\left <(\varepsilon +|\nabla u|^{n-2}) \nabla_k u, \beta |\nabla u|^{\beta-1} \nabla_k(|\nabla u|)  \partial_t u\right >\phi^n \,dv \,dt
\\
        &&-\int_{P_{2R_0}(x_0, t_0)}\left < (\varepsilon +|\nabla u|^{n-2}) \nabla_k u, |\nabla u|^{\beta} \partial_t u\right >\nabla_k\phi^n\,dv \,dt\\
        &&+ \int_{P_{2R_0}(x_0, t_0)}\left <  (\varepsilon +|\nabla u|^{n-2})A(u)(\nabla u, \nabla u), |\nabla u|^{\beta} \partial_t u\right >\phi^n \,dv \,dt \\
        &\leq&-\int_{B_{R_0}(x_0, t_0)}(\frac {\varepsilon} {2+\beta}|\nabla u|^{2+\beta} +\frac 1 {n+\beta} |\nabla u|^{n+\beta}) \phi^n (\cdot ,t_0)\,dv\\
        &&+ \int_{P_{2R_0}(x_0, t_0)}(\frac {\varepsilon} {2+\beta}|\nabla u|^{2+\beta} +\frac 1 {n+\beta} |\nabla u|^{n+\beta}) n\phi^{n-1} \partial_t\phi\,dv\,dt\\
       && +\frac a 2 \int_{P_{2R_0}(x_0, t_0)} (\varepsilon +|\nabla u|^{n-2})|\nabla u|^{\beta}  |\partial_t u|^2\phi^n \,dv \,dt\\
         && + \frac {\beta^2}{2a}\int_{P_{2R_0}(x_0, t_0)} (\varepsilon +|\nabla u|^{n-2})  |\nabla u|^{\beta-2} |\nabla_k u  \nabla_k(|\nabla u|)|^2)\phi^n \,dv \,dt\\
        &&+C \int_{P_{2R_0}(x_0, t_0)}(\varepsilon +|\nabla u|^{n-2}) |\nabla u|^{\beta+1}\phi^{n-1}  (\phi |\nabla u| +|\nabla \phi|)|\partial_t u|\,dv\,dt .
  \end{eqnarray*}
  This implies that
 \begin{eqnarray}\label{17}
        &&\qquad  a^2 \int_{P_{2R_0}(x_0, t_0)} (\varepsilon +|\nabla u|^{n-2})|\nabla u|^{\beta} |\partial_t u|^2\phi^n  \,dv \,dt \\
        &     \leq&  \beta^2
        \int_{P_{2R_0}(x_0, t_0)}(\varepsilon +|\nabla u|^{n-2})|\nabla u|^{\beta-2} |\nabla_k u  \nabla_k(|\nabla u|)|^2 \,dv \,dt\nonumber \\
        &+&C a \int_{P_{2R_0}(x_0, t_0)}(\varepsilon +|\nabla u|^{n-2}) |\nabla u|^{\beta }\phi^{n-1}  ( \phi |\nabla u|^2+|\nabla u|\,|\nabla \phi|)|\partial_t
u|\,dv\,dt\nonumber\\
        &&+\frac {C a} {2+\beta} \int_{P_{2R_0}(x_0, t_0)}(\varepsilon +|\nabla u|^{n-2}) |\nabla u|^{\beta+2 }\phi^{n-1} |\partial_t \phi|\,dv\,dt.\nonumber
        \end{eqnarray}
     Then it follows from using using  Young's inequality and (\ref{17}) that
        \begin{eqnarray}  \label{18}
         &&\quad -a(n-2)\int_{P_{2R_0}(x_0, t_0)}\left <|\nabla u|^{n-3}\nabla_l(|\nabla u|) \partial_t u, |\nabla u|^{\beta} \nabla_l u \right > \phi^n \,dv \,dt\\
         &\leq&  (n-2)\int_{P_{2R_0}(x_0, t_0)}  |\nabla u|^{n-2+\beta}(\frac  {a^2| \partial_t u|^2}{2\beta} +\frac {\beta} {2
         } |\nabla_l(|\nabla u|)  \nabla_l u |^2|\nabla u|^{-2})\phi^n \,dv \,dt\nonumber\\
         &\leq& \beta (n-2)  \int_{P_{2R_0}(x_0, t_0)} (\varepsilon +|\nabla u|^{n-2})|\nabla u|^{\beta-2} |\nabla_l (|\nabla u|) \nabla_lu|^2\phi^n\,dv \,dt\nonumber \\
         &+&C\frac a {\beta} \int_{P_{2R_0}(x_0, t_0)}(\varepsilon +|\nabla u|^{n-2}) |\nabla u|^{\beta+1}\phi^{n-1}[  (\phi |\nabla u|+|\nabla \phi|)|\partial_t u|+ \frac 1
{\beta}|\partial_t \phi|]\,dv\,dt.\nonumber
         \end{eqnarray}

  Multiplying (\ref {APH}) by $\phi^n\nabla \cdot (|\nabla u|^{\beta}\nabla u)$ and integrating by parts, we have
 \begin{eqnarray}\label{19}
        && \int_{P_{2R_0}(x_0, t_0)}\left <\nabla_k ( (\varepsilon +|\nabla u|^{n-2})\nabla_k u), \nabla_l(|\nabla u|^{\beta}\nabla_l u)\right >\phi^n  \,dv \,dt     \\
        &&=\int_{P_{2R_0}(x_0, t_0)}\left <(1-a+\varepsilon+a|\nabla u|^{n-2}) \partial_t u, \nabla_l(|\nabla u|^{\beta} \nabla_l u)\right >\phi^n  \,dv \,dt\nonumber\\
         &&-\int_{P_{2R_0}(x_0, t_0)}\left < (\varepsilon +|\nabla u|^{n-2})A(u)(\nabla u, \nabla u), \nabla_l(|\nabla u|^{\beta}\nabla_l u)\right >\phi^n \,dv
\,dt.\nonumber\\
         &&=\int_{P_{2R_0}(x_0, t_0)}\left <(1-a+\varepsilon+a|\nabla u|^{n-2}) \partial_t u, \nabla_l(|\nabla u|^{\beta} \nabla_l u)\right >\phi^n  \,dv \,dt\nonumber\\
         &&+\int_{P_{2R_0}(x_0, t_0)}\left < \nabla_l[ (\varepsilon +|\nabla u|^{n-2}) A(u)(\nabla u, \nabla u) ], |\nabla u|^{\beta}\nabla_l u\right >\phi^n \,dv
\,dt\nonumber\\
         &&+\int_{P_{2R_0}(x_0, t_0)}\left < (\varepsilon +|\nabla u|^{n-2})A(u)(\nabla u, \nabla u), |\nabla u|^{\beta}\nabla_l u \right >\nabla_l(\phi^n) \,dv
\,dt.\nonumber
         \end{eqnarray}
         The second term of the right-hand side of (\ref{19})  is a good one, but we need to analyze the first term of the right-hand side of (\ref{19}). In order to estimate the first term of the right-hand side,
using equation (\ref{APH}), we note that
          \[a|\partial_t u|\leq C (|\nabla^2 u|+|\nabla u|^2).\]
         Then,  integrating by parts and using (\ref {18}), we  have
       \begin{eqnarray}\label{20}
         &&\int_{P_{2R_0}(x_0, t_0)}\left <(1-a+\varepsilon+a|\nabla u|^{n-2}) \partial_t u, \nabla_l(|\nabla u|^{\beta} \nabla_l u)\right >\phi^n\,dv\,dt\\
         &=&-a(n-2)\int_{P_{2R_0}(x_0, t_0)}\left <|\nabla u|^{n-3}\nabla_l(|\nabla u|) \partial_t u, |\nabla u|^{\beta} \nabla_l u \right > \phi^n \,dv \,dt\nonumber \\
         &&-\int_{P_{2R_0}(x_0, t_0)}\left < (1-a+\varepsilon+a|\nabla u|^{n-2})  \nabla_l(\partial_t u),  |\nabla u|^{\beta} \nabla_l u \right > \phi^n \,dv \,dt \nonumber\\
         && -\int_{P_{2R_0}(x_0, t_0)}\left <(1-a+\varepsilon+a|\nabla u|^{n-2}) \partial_t u,  |\nabla u|^{\beta} \nabla_l u \right >\nabla_l(\phi^n )\,dv \,dt.\nonumber\\
        &\leq& \beta (n-2)  \int_{P_{2R_0}(x_0, t_0)} (\varepsilon +|\nabla u|^{n-2})|\nabla u|^{\beta-2} |\nabla_l (|\nabla u|) \nabla_lu|^2\phi^n\,dv \,dt\nonumber \\
        &&-  \int_{B_{2R_0}(x_0, t_0)}(\frac {1-a+\varepsilon} {2+\beta} |\nabla u|^{2+\beta} +\frac a {n+\beta}|\nabla u|^{n+\beta}) \phi^n (\cdot ,t_0)\,dv\nonumber\\
        && +\int_{P_{2R_0}(x_0, t_0)}(\frac {1-a+\varepsilon} {2+\beta} |\nabla u|^{2+\beta} +\frac a {n+\beta}|\nabla u|^{n+\beta}) \phi^{n-1}\partial_t\phi
\,dv\,dt\nonumber\\
        &&-\int_{P_{2R_0}(x_0, t_0)}\left <({1-a+\varepsilon}+a|\nabla u|^{n-2}) \partial_t u,  |\nabla u|^{\beta} \nabla_l u \right >\nabla_l(\phi^n )\,dv \,dt\nonumber
        \\&&+C \int_{P_{2R_0}(x_0, t_0)}(\varepsilon +|\nabla u|^{n-2}) |\nabla u|^{\beta+1}\phi^{n-2}[\frac 1 {\beta^2} |\nabla \phi|^2 + \frac 1 {\beta^2}\phi |\partial_t
\phi|]\,dv\,dt\nonumber\\
        &&+\int_{P_{2R_0}(x_0, t_0)}(\varepsilon +|\nabla u|^{n-2}) |\nabla u|^{\beta}\left  (\frac 1 4|\nabla^2 u|^2+ C|\nabla u|^{4}\right ) \phi^n\,dv\,dt.
\nonumber\end{eqnarray}

 To  estimate the first term of the left-hand side of (\ref{18}),  integrating by parts twice and using the Ricci formula yield that
\begin{eqnarray}\label{21}
        &&\int_{P_{2R_0}(x_0, t_0)}\left <\nabla_k ( (\varepsilon+|\nabla u|^{n-2})\nabla_k u), \nabla_l(|\nabla u|^{\beta}\nabla_l u) \right >\phi^n\,dv\,dt\\
         &&= \int_{P_{2R_0}(x_0, t_0)}\left <\nabla_l ( (\varepsilon+|\nabla u|^{n-2})\nabla_k u), \nabla_k(|\nabla u|^{\beta}\nabla_l u) \right >\phi^n\,dv\,dt\nonumber\\
         &&\quad +  \int_{P_{2R_0}(x_0, t_0)}\left <R_M\# ( (\varepsilon+|\nabla u|^{n-2})\nabla_k u),  |\nabla u|^{\beta}\nabla_l u  \right >\phi^n\,dv\,dt\nonumber\\
         && \quad +  \int_{P_{2R_0}(x_0, t_0)}\left <  \nabla_l( (\varepsilon+|\nabla u|^{n-2})\nabla_k u), |\nabla u|^{\beta}\nabla_l u  \right >\nabla_k
\phi^n\,dv\,dt\nonumber\\
      && \quad -  \int_{P_{2R_0}(x_0, t_0)}\left <  \nabla_k( (\varepsilon+|\nabla u|^{n-2})\nabla_k u), |\nabla u|^{\beta}\nabla_l u  \right >\nabla_l
\phi^n\,dv\,dt.\nonumber
  \end{eqnarray}
  Moreover, we note that
\begin{eqnarray}\label{22}
        && \quad \int_{P_{2R_0}(x_0, t_0)}\left <\nabla_l ((\varepsilon+|\nabla u|^{n-2})\nabla_k u), \nabla_k(|\nabla u|^{\beta}\nabla_l u) \right >\phi^n\,dv\,dt\\
         &&= \int_{P_{2R_0}(x_0, t_0)}(\varepsilon+|\nabla u|^{n-2})|\nabla u|^{\beta}|\nabla^2 u|^2\phi^n\,dv\,dt\nonumber\\
          &&+(n-2+\beta ) \int_{P_{2R_0}(x_0, t_0)}(\varepsilon+|\nabla u|^{n-2})|\nabla u|^{\beta}|\nabla (|\nabla u|)|^2\phi^n\,dv\,dt \nonumber\\
          &&+\beta (n-2) \int_{P_{2R_0}(x_0, t_0)} (\varepsilon+|\nabla u|^{n-2})|\nabla u|^{ \beta -2}|\nabla_l (|\nabla u| ) \nabla_lu|^2  \phi^n\,dv\,dt \nonumber
  \end{eqnarray}
  and
  \begin{eqnarray} \label{23}
  &&\int_{P_{2R_0}(x_0, t_0)}\left <  \nabla_l( (\varepsilon+|\nabla u|^{n-2})\nabla_k u), |\nabla u|^{\beta}\nabla_l u  \right >\nabla_k \phi^n\,dv\,dt\\
  &=& \int_{P_{2R_0}(x_0, t_0)}\left <  \nabla_l( |\nabla u|^{n-2})\nabla_k u, |\nabla u|^{\beta}\nabla_l u  \right >\nabla_k \phi^n\,dv\,dt \nonumber\\
  && +\frac 12 \int_{P_{2R_0}(x_0, t_0)}   (\varepsilon+|\nabla u|^{n-2})|\nabla u|^{\beta}  \nabla_k |\nabla u|^2   \nabla_k \phi^n\,dv\,dt\nonumber\\
  &\leq & \frac  {n-2+\beta}2 \int_{P_{2R_0}(x_0, t_0)}(\varepsilon+|\nabla u|^{n-2})|\nabla u|^{\beta}|\nabla (|\nabla u|)|^2\phi^n\,dv\,dt\nonumber\\
    &&+ \frac C{n-2+\beta}\int_{P_{2R_0}(x_0, t_0)} (\varepsilon+|\nabla u|^{n-2})|\nabla u|^{\beta +2}  \phi^{n-2}|\nabla \phi |^2\,dv\,dt. \nonumber
    \end{eqnarray}

  Combining (\ref{19})-(\ref{22}) with (\ref{23}), we have
  \begin{eqnarray}\label{24}
       &&  \frac 12 \int_{P_{2R_0}(x_0, t_0)}(\varepsilon+|\nabla u|^{n-2})|\nabla u|^{\beta}|\nabla^2 u|^2 \phi^n\,dv\,dt
       \\
       && + \frac{ (n-2+\beta ) }2\int_{P_{2R_0}(x_0, t_0)}(\varepsilon+|\nabla u|^{n-2})|\nabla u|^{\beta}|\nabla (|\nabla u|)|^2\phi^n\,dv\,dt\nonumber
       \\
       &&+  \int_{B_{2R_0}(x_0)}(\frac 1 {2+\beta} |\nabla u|^{2+\beta} +\frac a {n+\beta}|\nabla u|^{n+\beta}) \phi^n (\cdot ,t_0)\,dv\nonumber \nonumber\\
          &\leq&   C\int_{P_{2R_0}(x_0, t_0)}(\frac 1 {2+\beta} |\nabla u|^{2+\beta} +\frac 1 {n+\beta}|\nabla u|^{n+\beta}) \phi^{n-1}|\partial_t\phi |\,dv\,dt\nonumber\\
          &+& C\frac 1 {\beta}\int_{P_{2R_0}(x_0, t_0)} (\varepsilon+|\nabla u|^{n-2})|\nabla u|^{\beta +2}  \phi^{n-2}|\nabla \phi |^2\,dv\,dt \nonumber\\
            &+&C\int_{P_{2R_0}(x_0, t_0)}(\varepsilon+|\nabla u|^{n-2}) (|\nabla u|^{2+\beta}+|\nabla u|^{4+\beta}) \phi^n\,dv\,dt,\nonumber
  \end{eqnarray}
  where  thhe third term of the right hand side of (\ref{19}),(\ref{20}) and the last term of the right-hand side of (\ref{22}) are canceled by   using equation (\ref{APH}).

On the other hand, by using the H\"older and Sobolev inequalities, we have
  \begin{eqnarray}\label {25}
        &&\quad\qquad  \int_{t_0-4R_0}^{t_0}\int_{B_{2R_0}(x_0)} |\nabla u|^{n+2+\beta}\phi^n\,dv\,dt\\
        &&\leq \int_{t_0-4R_0^2}^{t_0}\left (\int_{B_{2R_0}(x_0)} |\nabla u|^{n}\,dv\right )^{\frac 2 n}\left (\int_{B_{2R_0}(x_0)} |\nabla u|^{\frac {n
(n+\beta)}{n-2}}\phi^{\frac {n^2}{n-2}}\,dv\right )^{\frac{n-2}n}\,dt\nonumber\\
        &&\leq C \varepsilon_1^{\frac 2 n} \int_{t_0-4R_0^2}^{t_0}  \int_{B_{2R_0^2}(x_0)}|\nabla ( |\nabla u|^{ \frac {n+\beta}2} \phi^{n/2})|^2\,dv \,dt\nonumber\\
        &&\leq C \varepsilon_1^{\frac 2 n} \int_{t_0-4R_0^2}^{t_0}  \int_{B_{2R_0^2}(x_0)}((n+\beta)^2|\nabla^2  u|^{2}|\nabla u|^{n-2+\beta}\phi^{n} +|\nabla u|^{n+\beta}
|\nabla \phi|^2\phi^{n-1})\,dv \,dt.\nonumber
  \end{eqnarray}
 Choosing $\varepsilon_1$ (depending on $\beta$ here) sufficiently small yields
\begin{eqnarray*}
  && \int_{t_0-4R_0}^{t_0}\int_{B_{2R_0}(x_0)} (|\nabla u|^{n+2+\beta}+(\varepsilon+|\nabla u|^{n-2})|\nabla u|^{\beta}|\nabla^2 u|^2 ) \phi^n dv\,dt\\
    && \leq C\int_{P_{2R_0}(x_0, t_0)}(1+ |\nabla \phi|^2 +|\partial_t\phi |)|\nabla u|^{n+\beta} \,dv\,dt.
    \end{eqnarray*}
    This proves our claim. \end{proof}
Since the constant $\varepsilon_1$ depends on $\beta$ in Lemma \ref {higher-regularity},    we have to get an improved estimate to obtain the gradient estimate in the
following:

\begin{lem} \label{boundness} Let $u$ be    a smooth solution  to the  flow equation (\ref{APH}). There exists a positive constant
    $\varepsilon_0<i(M)$ such that  if for some $R_0$ with $0<R_0<\min
\{\varepsilon_0, \frac {t_0^{1/2}}2\}$  the inequality
\[\sup_{t_0-4R_0^2\leq t<t_0}\int_{B_{2R_0}(x_0)}|\nabla u|^{n} \,dv
    < \varepsilon_0\] holds,
we have
\[\sup_{P_{R_0}(x_0,t_0)} |\nabla u|^{n} \leq C  R_0^{-n},\]
where  $C$   is a constant independent of $\varepsilon$, $a$ and $R_0$.
\end{lem}
\begin{proof} Let $\phi=\phi (x,t)$ be a cut-off function with support in  $B_{R}(x_0)\times [t_0-R_0^2, t_0+R_0^2]$ such that $\phi =1$ in $B_{R}(x_0)\times [t_0-\rho^2,
t_0+\rho^2]$, $|\nabla \phi|\leq \frac C{R-\rho}$, $|\partial_t \phi|\leq \frac 1 {(R-\rho)^2}$ and $|\phi|\leq 1$ in  $B_{R}(x_0)\times [t_0-R^2, t_0+R^2]$.
 For this new cut-off function $\phi$, the same proof of (\ref{24}) gives
\begin{eqnarray}\label{25.1}
       &&  \frac 12 \int_{P_{R}(x_0, t_0)}(\varepsilon+|\nabla u|^{n-2})|\nabla u|^{\beta}|\nabla^2 u|^2 \phi^n\,dv\,dt
       \\
       && + \frac{ (n-2+\beta ) }2\int_{P_{R}(x_0, t_0)}(\varepsilon+|\nabla u|^{n-2})|\nabla u|^{\beta}|\nabla (|\nabla u|)|^2\phi^n\,dv\,dt\nonumber
       \\
       &&+\sup_{t_0-R^2\leq s\leq t_0}  \int_{B_{R}(x_0)}(\frac 1 {2+\beta} |\nabla u|^{2+\beta} +\frac a {n+\beta}|\nabla u|^{n+\beta}) \phi^n (\cdot ,s)\,dv\nonumber
\nonumber\\
          &\leq&   C\int_{P_{R}(x_0, t_0)}(\frac 1 {2+\beta} |\nabla u|^{2+\beta} +\frac 1 {n+\beta}|\nabla u|^{n+\beta}) \phi^{n-1}|\partial_t\phi |\,dv\,dt\nonumber\\
          &+& C\frac 1 {\beta}\int_{P_{R}(x_0, t_0)} (\varepsilon+|\nabla u|^{n-2})|\nabla u|^{\beta +2}  \phi^{n-2}|\nabla \phi |^2\,dv\,dt \nonumber\\
            &+&C\int_{P_{R}(x_0, t_0)}(\varepsilon+|\nabla u|^{n-2}) (|\nabla u|^{2+\beta}+|\nabla u|^{4+\beta}) \phi^n\,dv\,dt,\nonumber
  \end{eqnarray}
  Using H\"older's and Sobolev's inequalities with (\ref{25.1}), we have
  \begin{eqnarray*}
        &&\quad \int_{t_0-\rho^2}^{t_0}\int_{B_{\rho}(x_0)} |\nabla u|^{(n+\beta) (1+\frac 2n \frac {\b+2}{\b +n} )}\,dv\,dt\\
        &&\leq \int_{t_0-\rho^2}^{t_0}\left ( \int_{B_{R}(x_0)} |\nabla u|^{\b +2}\phi^n\,dv\right )^{\frac 2 n}\left (\int_{B_{R}(x_0)} |\nabla u|^{\frac {n
(n+\beta)}{n-2}}\phi^{\frac {n^2}{n-2}}\,dv\right )^{\frac{n-2}n}\,dt\\
        &&\leq C  \sup_{t_0-R^2\leq t\leq t_0}\left (\int_{B_{R}(x_0)} |\nabla u|^{\b +2}\phi^n\,dv\right )^{\frac 2 n}\int_{t_0-R^2}^{t_0}  \int_{B_{R}(x_0)}|\nabla ( |\nabla
u|^{ \frac {n+\beta}2} \phi^{n/2})|^2\,dv \,dt\\
        &&\leq C\left (\int_{P_{R}(x_0, t_0)}( |\nabla u|^{2+\beta} +|\nabla u|^{n+\beta}) \phi^{n-2}(|\partial_t\phi |+|\nabla \phi |^2)\,dv\,dt\right .\nonumber\\
            && \quad +\left .\b \int_{P_{R}(x_0, t_0)}(\varepsilon+|\nabla u|^{n-2}) (|\nabla u|^{2+\beta}+|\nabla u|^{4+\beta}) \phi^n\,dv\,dt\nonumber\right )^{1+2/n}
  \end{eqnarray*}

  Next, we follow \cite {Hung} to process a Moser's iteration (e.g. see \cite {GT}).

  Set $R=R_{k}=R_0(1+2^{-k})$, $\rho =R_{k+1}= R_0(1+2^{-1-k})$, $\b=\b_k=\theta^k(d_0-2n)+n-2$ and $\theta =1+2/n$ with $d_0>2n$.
  \[d_k=n+\beta_k+2=\theta^k(d_0-2n)+2n,\quad d_{k+1}=(n+\b_k )\left (1+\frac 2 n \frac{\b_k+2}{\b_k+n}\right )=\theta d_k-4.\]
Then
\begin{eqnarray*}
        &&\quad \int_{P_{k+1}}(1+|\nabla u|^{d_{k+1}}) \,dv\,dt\leq C4^{k\theta} \left (\int_{P_k}(1+ |\nabla u|^{d_{k}})\,dv\,dt\right )^{\theta}.
  \end{eqnarray*}
  Set
  \[I_k= \left (\int_{P_k} (1+|\nabla u|^{d_{k}})\,dv\,dt\right )^{\frac 1{\theta^k}}. \]
Applying an iteration, we have

\[I_{k+1}\leq C^{\frac 1{\theta^{k+1}}}4^{\frac {k}{\theta^{k}}}I_k\leq C^{\sum_{k=1}^{\infty}\frac 1{\theta^{k+1}}}4^{\sum_{k=1}^{\infty}\frac {k}{\theta^{k}}}I_0\leq \tilde
CI_0.
 \]
Therefore, noting $d_k=\theta^k(d_0-2n)+2n$ for all $k\geq 1$, we have
\begin{eqnarray*}
        &&\left (\int_{P_{R_0}} |\nabla u|^{\theta^{k+1}(d_0-2n)}\,dv\,dt \right )^{\frac 1 {\theta^{k+1}(d_0-2n)}}\\
        &&\leq \left (C\int_{P_{k+1}} (1+|\nabla u|^{d_{k+1}}) \,dv\,dt\right )^{\frac 1 {\theta^{k+1}(d_0-2n)}}\\
        &\leq &C^{\frac 1 {\theta^{k+1}(d_0-2n)}}  (\tilde CI_0)^{\frac 1 {(d_0-2n)}} \leq C(u_0, R_0).
        \end{eqnarray*}
This implies that
$|\nabla u|$ is bounded in $P_{R_0}$.
\end{proof}

\begin{lem}
    Let $u:M\to N$ be a smooth solution to the flow equation (\ref {APH}).
   There is a small constant $\varepsilon_0>0$ such that if  the inequality
\[\sup_{t_0-T'\leq t<t_0}\int_{B_{2R_0}(x_0)}|\nabla u|^{n} \,dv\,dt
    < \varepsilon_0\] holds for some positive $R_0$,
then $\|u\|_{C^{0,\alpha}}(P_{R_0}((x_0,t_0)))$ is uniformly bounded in $\varepsilon$.
\end{lem}
\begin{proof} Using the above Lemma \ref{boundness} , $|\nabla u|$ is bounded by a constant $C$.
By a similar proof of the local energy inequality, we have
\begin{eqnarray*}
 &&\int_{P_R(z_0)} (\varepsilon +|\nabla u|^{n-2})|\frac {\partial u}{\partial t} |^2\,dv\,dt \\
  &&\leq  C \sup_{t_0-R^2\leq t\leq t_0}E_{\varepsilon}(u(t); B_{2R}(x_0))  \leq CR^n.
    \end{eqnarray*}
Set $u_{z_0,R}=\int_{P_{R}(z_0)}u(x,t)\,dz$.
By a variant of the  Sobolev-Poincare inequality, we have
\begin{eqnarray*}
 \int_{P_{R}(z_0)} |u-u_{z_0,R}|^2\,dv\,dt&\leq & C\left [R^{2}\int_{P_{R}(z_0)}|\nabla u|^{2}\,dv\,dt+R^{4}\int_{P_{R }(z_0)}|\partial_t u|^{2}\,dv\,dt  \right ] \\
&\leq & CR^{n+4}
 \end{eqnarray*}
for all $R\leq R_0/2$.
This implies that $u(x,t)$ is H\"older continuous near $(x_0,t_0)$.

\end{proof}

\begin{thm}\label{Theorem 6} For any $u_0\in W^{1,n}(M, N)$, there
exists a local  solution $u: M\times [0,T_0]\to N$ of the flow equation (\ref{NA}) with initial value $u_0$ for a constant $T_0$ satisfying
 \begin{eqnarray} \label {second}
  &&\int_{0}^{T_0}\int_{M} (|\nabla u|^{n+2}+|\nabla^2 u|^2 |\nabla u|^{n-2})\,dv\,dt\\
  &&
  \leq C E_{n}(u_0)+ C(1+T_0R_0^{-2})E_{n}(u_0).\nonumber
    \end{eqnarray}
 \end{thm}
 \begin{proof}
Since $u_0\in W^{1,n}(M, N)$ can be approximated by maps in $C^{\infty}(M, N)$, we assume that $u_0$ is smooth without loss of generality.
Let $u_{a,\varepsilon}$ be a solution of that equation (\ref{APH}) with smooth initial value $u_0$.
Note that equation (\ref{APH}) is equivalent to
\begin{eqnarray}\label{26}
\frac {\partial   u_{a,\varepsilon}^{\beta}}{\partial t}&=& \frac 1{( 1-a+\varepsilon+a|\nabla   u_{a,\varepsilon}|^{n-2})} \frac{1}{\sqrt{|g|}}\frac{\partial}{\partial
x_{i}}\left[ ( \varepsilon
+ |\nabla u_{a,\varepsilon}|^{n-2})g^{ij}\sqrt{|g|}\frac{\partial}{\partial x_{j}}  u_{a,\varepsilon}^{\beta}\right]\nonumber\\
&&+\frac {(\varepsilon  +|\nabla
  u_{a,\varepsilon}|^{n-2})A^{\beta}(u_{a,\varepsilon})(\nabla   u_{a,\varepsilon},\nabla   u_{a,\varepsilon})}{ (1-a+\varepsilon+a|\nabla   u_{a,\varepsilon}|^{n-2})
}\nonumber\\
&:=&\sum_{i,k,\a}{b^{\a\b}_{ij}}_a(\nabla  u_{a,\varepsilon} )\frac{\partial^2   u_{a,\varepsilon}^{\a}}{\partial x_{i}\partial x_j} + f(u_{a,\varepsilon},\nabla
u_{a,\varepsilon}),
    \end{eqnarray}
    where
    \[{b^{\a\b}_{ij}}_a(\nabla u_{a,\varepsilon}  )=\frac {\varepsilon +|\nabla u_{a,\varepsilon}|^{n-2}}{(1-a+\varepsilon+a|\nabla   u_{a,\varepsilon}|^{n-2})}\left
(g^{ij}\delta^{\a\b}+
    \frac {(n-2) |\nabla u_{a,\varepsilon}|^{n-4}\partial_{x_i}u_{a,\varepsilon}^{\a}\partial_{x_j}u_{a,\varepsilon}^{\b}}{\varepsilon +|\nabla u_{a,\varepsilon}|^{n-2}}
\right ).\]
For a fixed  parameter $\varepsilon$, (\ref{26}) is a parabolic system, so there is a local smooth solution $u_{a, \varepsilon}$ to
the rectified gradient flow (\ref{APH}) with smooth initial value $u_0$ in $[0, T_{a, \varepsilon})$ for a maximal existence time $T_{a, \varepsilon}$.

  For a fixed  $\varepsilon>0$, there is a constant $\tilde T>0$, depending on the bound of $u_0$ and its higher derivatives, such that $\tilde T\leq T_{a, \varepsilon}$ for
all $a\in [0,1]$.  In order to prove the local existence of (\ref{NA}), we need to show that there is a uniform constant $T_0>0$, depending only on $E_n(u_0)$, such that
$T_{a,\varepsilon}\geq T_0$ for all $\varepsilon >0$ and $a\in [0,1]$.  Since $T_{a,\varepsilon}$ is the maximal existence time of the smooth solution
$u_{\varepsilon}$ of the flow (\ref{APH}),  it follows from using the same proof of  Theorem 1 in \cite {Hung} (Section 2.5) that there is a constant $T_0>0$, depending only
on $E_n(u_0)$, $\varepsilon_0$ and $R_0$, such that for $t\leq T_0$, we have
\begin{align}  \int_{B_{R_0}(x_0) }e_{\varepsilon}(u_{a,\varepsilon})(\cdot ,t)\,dv\leq &\,\int_{B_{2R_0}(x_0)} e_{\varepsilon}(u_0) \,dv
+  \frac {Ct}{R_0^n}\,\left (\int_M e_{\varepsilon}(u_0)\,dv\right )^{1-\frac 1 n} <\varepsilon_0. \end{align}
 If $\tilde T\leq T_0$, then  it follows from  using Lemma \ref{boundness}  that $\nabla   u_{a,\varepsilon}$ is bounded in $M\times [0,\tilde T]$ by the norm $\|\nabla
u_0\|_{L^n(M)}$ and hence $ f(u_{a,\varepsilon},\nabla u_{a,\varepsilon})$ is bounded.   By the PDE theory, $\nabla  u_{a,\varepsilon} (x,t)$ is  continuous in $a\in [0,1]$
for any $t\leq\tilde T<T_{\varepsilon}$.  For any $\tilde \delta >0$, there is a $\eta>0$ such that for any two $a, a_0\in [0, 1]$ with $|a-a_0|<\eta$, we have
    \[|{b^{\a\b}_{ij}}_a(\nabla  u_{a,\varepsilon} )(x,t)-{b^{\a\b}_{ij}}_{a_0}(\nabla  u_{a_0,\varepsilon} )(x,t)|\leq \tilde\delta.\]

We assume that $ \nabla  u_{a_0,\varepsilon} (x,t)$ is H\"older continuous in $M\times [\frac {\tilde T}4,\tilde T]$, with its H\"older norm depending only on the bound of
$\nabla  u_{a_0,\varepsilon} (x,t)$.  In fact, this is known for $a_0=0$ (see \cite {Hung}).
Notting
  \begin{eqnarray}
&&\frac {\partial   u_{a,\varepsilon}^{\beta}}{\partial t} -{b^{\a\b}_{ij}}_{a_0}(\nabla  u_{a_0,\varepsilon} )\frac{\partial^2   u_{a,\varepsilon}^{\a}}{\partial
x_{i}\partial x_j} \nonumber\\
&=&\left ({b^{\a\b}_{ij}}_a(\nabla  u_{a,\varepsilon} ) -{b^{\a\b}_{ij}}_{a_0}(\nabla  u_{a_0,\varepsilon} )\right ) \frac{\partial^2   u_{a,\varepsilon}^{\a}}{\partial
x_{i}\partial x_j} + f(u_{a,\varepsilon},\nabla u_{a,\varepsilon}),\nonumber
    \end{eqnarray}
   we apply the $L^p$-estimate to obtain that
 \begin{eqnarray}
&&\int_{P_{R/2} (x, \tilde T)} |\frac {\partial   u_{a,\varepsilon} }{\partial t}|^pdvdt+\int_{P_{R/2} (x, \tilde T)} |\nabla^2   u_{a,\varepsilon} |^pdvdt \\
&\leq & C\tilde \delta \int_{P_{R} (x, \tilde T)}   |\nabla^2   u_{a,\varepsilon} |^pdvdt+ C \int_{P_R(x, \tilde T)} (|f(u_{a,\varepsilon},\nabla
u_{a,\varepsilon})|^p+|u_{a,\varepsilon}|^p)\,dv dt.\nonumber
\end{eqnarray}

 By a covering argument od $M$ and choosing $\tilde\delta$ sufficiently small with $C\tilde\delta <\frac 14$, we have
 \begin{eqnarray}
&&
 \int_{M\times [\frac 12 \tilde T, \tilde T]} |\frac {\partial   u_{a,\varepsilon} }{\partial t}|^pdxdt+\frac 12  \int_{M\times [\frac 12 \tilde T, \tilde T]} |\nabla^2
u_{a,\varepsilon} |^pdvdt\\
 &\leq &C \int_{M\times [\frac 14 \tilde T, \frac 12 \tilde T]} (|f(u_{a,\varepsilon},\nabla u_{a,\varepsilon})|^p+|u_{a,\varepsilon}|^p+ |\nabla^2   u_{a,\varepsilon}
|^p)\,dvdt\leq C(\tilde T). \nonumber
 \end{eqnarray}

By the Sobolev imbedding theorem of parabolic version, $\nabla u_{a,\varepsilon}$   is also H\"older continuous, depending on $C(\tilde T)$,  uniformly for all $a\in [0,1]$
and therefore $u_{a,\varepsilon} $ is smooth across to $\tilde T\geq T_0$ for all $a\in [0,1]$.
Therefore, for each fixed $\varepsilon>0$, there is a smooth solution of the flow (\ref{APH}) in $[0, T_0]$   satisfying
\begin{eqnarray}
  &&\int_{0}^{T_0}\int_{M} |\nabla u_{a,\varepsilon}|^{n+2}+|\nabla^2 u_{a,\varepsilon}|^2(\varepsilon +|\nabla u_{a,\varepsilon}|^{n-2})\,dv\,dt\\
  &&
  \leq C E_{n,\varepsilon}(u_0)+ C(1+T_0R_0^{-2})E_{n,\varepsilon}(u_0).\nonumber
    \end{eqnarray}
 As $\varepsilon\to 0$, $u_{a,\varepsilon}$ converges to a map $u$, which is a solution of the flow equation (\ref{NA}) satisfying (\ref {second})  using Lemmas
\ref{Lemma 2.3}-\ref{boundness}.
    \end{proof}

Using above results, we can prove Theorem \ref{Theorem 1}:

\begin{proof}[Proof of Theorem \ref{Theorem 1} ]

By Theorem \ref{Theorem 6}, there is a local  solution to the flow equation (\ref{NA}) satisfying (\ref {second}). Then, the solution can be extended to $M\times
[0,T_1)$ for a maximal time
$T_1$ such that as $t\to T_1$, there are finite   singular points $\{x^{j,1}\}_{j=1}^{l_1}$; i.e.
 there is a constant $\varepsilon_0>0$ such that  each singular point  $x^{j,1}$  is characterized  by the condition
\[ \liminf_{t \to T_1}  E_n(u(x,t); B_R(x^j)) \geq \varepsilon_0 \] for any $R\in (0, R_0]$. In fact, the finiteness of singular points comes from using a similar local energy
inequality  to Lemma \ref {Local} (see \cite{St}). Then, we continue the above procedure at the initial time $T_1$ to prove existence of a solution of the flow
(\ref{NA}) in $[T_1,T_2)$ for a second blew up time $T_2$. By induction, we complete a proof.
\end{proof}

\section{Energy identity and neck-bubble decompositions}

In this section, let $u(x,t)$ be a  solution of the  rectified $n$-flow (\ref{NA}) in $M\times [0,T_1)$ in Theorem \ref{Theorem 1}.
Consider now a sequence of $\{u(x, t_i)\}$  as $t_i \to T_1\leq \infty$. Then they have uniformly bounded energy; i.e.
$E_n(u(t_i);M) \leq E_n(u_0; M)$. As $t_i \to  T_1$,   $u(x, t_i)$
 converges  to a  map $u_{T_1}$ strongly  in $W_{loc}^{1,n+1}(M\backslash\{x^1,\cdots, x^l\} )$ with finite integer $l$. At each singularity $x^j$, there is a
$R_0>0$ such that there is no other singularity inside $B_{R_0}(x^j)$. Moreover,
 there is a constant $\varepsilon_0>0$ such that  each singular point  $x^j$ for $j=1,...,l$ is characterized  by the condition
\[ \liminf_{i \to \infty}  E_n(u_i; B_R(x^j)) \geq \varepsilon_0 \] for any $R\in (0, R_0]$.
 Then there is a $\Theta>0$ such that as $t_i\to T_1$
\begin{equation}\label{limit}   |\nabla u (x,t_i)|^{n} dv\to \Theta \delta_{x^j} + |\nabla u_{T_1}|^{n}  dv,\end{equation}
 where
$\delta_{x^j}$ denotes the Dirac mass at the singularity $x^j$.

In order to establish  the energy identity of the sequence $\{u(x, t_i)\}_{i=1}^{\infty}$, we need to get the neck-bubble decomposition. We recall the  removable singularity
theorem of $n$-harmonic maps \cite {DF} and the gap theorem: there is a constant $\varepsilon_{g}>0$ such that if $u$ is a $n$-harmonic map on $S^n$ satisfying
$\int_{S^n}|\nabla u|^n<\varepsilon_{g}$, then $u$ is a constant on $S^n$.
For completeness, we  give a detailed proof on constructing the  bubble-neck decomposition   by following the idea of Ding-Tian \cite{DT} (also \cite{Parker}).

\medskip\noindent{\bf Step 1.} To find a maximal (top) bubble at the level one  (first re-scaling).

Since $ u (x, t_i) \to u_{T_1}$ regularly in $B_{R_0}(x^j)$ away from $x^j$,
where $u_{T_1}$ is a map in $W^{1,n}(M, N)$.
Since $x^j$ is a concentration point, we find   such that as $t_i\to T_1$,
\[ \max_{x\in B_{R_0}(x^j),\, T_1-\delta\leq t\leq t_i} | \nabla u (x,t)|\to \infty,  \quad  r_{i,1}=\frac 1 {\max_{x\in B_{R_0}(x^j),\,T_1-\delta \leq t\leq t_i}   | \nabla u
(x,t)|}\to 0 \] for a small $\delta >0$.
In  the neighborhood of the singularity $x^j$,   we define the rescaled map
\[  \tilde u_i(\tilde x,\tilde t):=  u_i(x_j+ r_{i,1}\tilde x, t_i+ (r_{i,1})^2 \tilde t).\]
Then $\tilde u_i(x,t)$ satisfies
\begin{align}
  \label{HA}&  ((r_{i,1})^{n-2}(1-a+\varepsilon ) +a|\nabla \tilde u|^{n-2}) \frac {\partial \tilde u}{\partial \tilde t} \\
  &= \mbox{div}\left (  |\nabla
\tilde u|^{n-2} {\nabla} \tilde u\right )
+|\nabla
\tilde u|^{n-2}A(\tilde u)(\nabla\tilde u,\nabla \tilde u)  \nonumber\end{align}   in  $B_{R_0r^{-1}(0)}\times [-1, 0]$
and   \begin{align} &\int_{-1}^{0} \int_{B_{R_0 (r_{i,1})^{-1}}(0)}  (r_{i,1}^{n-2}(1-a+\varepsilon ) +a|\nabla \tilde u|^{n-2}) |\frac {\partial \tilde u}{\partial \tilde t}|^2\,d\tilde
v\,d\tilde t
\\&\leq \int^{t_i}_{t_i-(r_{i,1})^2} \int_{M}  ((1-a+\varepsilon )+a |\nabla u|^{n-2} )\left |\frac{\partial u}{\partial  t}\right |^2\,dv\,dt\to 0. \nonumber\end{align}
Therefore, there is a $\tilde t\in (-1,0)$ such that
\begin{equation}\label{}
\int_{B_{R_0 (r_{i,1})^{-1}}(0)}  (r_{i,1}^{n-2}(1-a+\varepsilon ) +a|\nabla \tilde u|^{n-2}) |\frac {\partial \tilde u}{\partial t}|^2(\cdot ,\tilde t) \,d\tilde v\to 0.
\end{equation}

Using Lemma \ref{Local}, it can be shown that as
$i\to\infty$,
 \begin{equation}\label{limit}   |\nabla u (x, t_i+r_{i,1}^2 \tilde t_i)|^{n} dv\to \Theta \delta_{x_j} + |\nabla u_{T_1}|^{n}  dv.\end{equation}
 For simplicity, we set
 \[ u_i(x):=u(x, t_i+r_{i,1}^2 \tilde t_i)\, \mbox { for }x\in B_{R_0}(x^j),\quad \tilde u_i(\tilde x):=u(x^j +r_{i,1}\tilde x, t_i+r_{i,1}^2 \tilde t_i).\]
Since $|\nabla \tilde u_i (\tilde x)|\leq 1$ for all $\tilde x\in B_{R_0r_{i,1}^{-1}}(0)$, $\tilde u_i$ sub-converges  to an $n$-harmonic map $\omega_{1,j}$ locally in
$C^{1,\a}(\R^n, N)$  as $i\to \infty$, and $\omega_{1,j}$ can be extended to an $n$-harmonic map on $S^n$ (see \cite {DF}) and is nontrivial due to (\ref {limit}). We call
$\omega_{1,j}$ the first bubble at the singularity $x^j$, which satisfies
\begin{equation}\label{First} E_n(\omega_{1,j}; \R^n) =\lim_{R \to \infty} \lim_{i \to \infty} E_n(\tilde u_i; B_{R}(0))=
\lim_{R \to \infty}\lim_{i \to \infty}E_n( u_i; B_{Rr_{i,1}}(x^j)). \end{equation}

\medskip\noindent{\bf Step 2.} To find out new bubbles at the second level (second re-scaling).

 Assume that for a fixed small constant $\varepsilon>0$ (to be chosen later),  there exist two positive constants $\delta_0$ and   $R_0$ with $R_0r_{i,1} <4 \delta_0$ such
that for all $i$ sufficiently large, we have
\begin{equation}\label{small} \int_{B_{2r} \setminus
B_r (x^j)}{|\nabla  u(t_i)|^n dV} \leq \varepsilon \end{equation}
 for all $r \in (\frac{Rr_{i,1}}{2},2\delta)$, and for all $R\geq R_0$ and $\delta\leq \delta_0$.

If  (\ref{small}) is   true,  it follows from (\ref{Local})  and (\ref{First})  that
 \begin{align*}\label{neckenergy}
\lim_{i \to \infty} E_n (u_i; B_{R_0}(x^j))=& E_n(u_{T_1}; B_{R_0}(x^j) )+   E_n(\omega_{1,j}; \R^n)
 \\
&+ \lim_{R \to \infty} \lim_{\delta \to 0}\lim_{i \to \infty}  E_n(u_i; B_\delta \setminus B_{Rr_{i,1}}(x^j)). \nonumber
\end{align*}
 In this case, this means that there is only single bubble $\omega_{1,j}$ around $x^j$.

 If the assumption (\ref{small}) is not true, then
 for any two constants $R$ and  $\delta$ with $Rr_{i,1}< 4\delta$, $\delta\leq\delta_0$ and $R\geq R_0$, there is a number $r_i\in  (\frac{Rr_{i,1}}{2},2\delta)$  such that
 \begin{equation}\label{Fact}
  \lim_{i \to \infty}  \int_{B_{2r_i}\backslash B_{r_i}(x^{j})}|\nabla  u_i|^n\,dV  >\varepsilon .\end{equation}
 Since there is a uniformly  energy bound
$K=nE_n(u_0;M)$, i.e. $\int_M|\nabla  u_i|^n\,dV \leq K$, and $\varepsilon$ is a fixed constant,   we remark that there is no infinitely number of above   $r_{i}\in
(0,\delta_0)$ with disjoint annuluses ${B_{2r_i}\backslash B_{r_i}(x^{j})}$ satisfying (\ref{Fact}). If $\liminf_{i\to \infty}r_{i}> 0$, it can be ruled out by choosing
$\delta_0$ sufficiently small, so we assume that $\lim_{i\to\infty}r_{i}=0$.
Similarly, if $
\limsup_{i \to \infty}\frac {r_{i}}{r_{i,1}}<\infty $, it can be rule out choosing $R_0$ sufficiently large since $\tilde u_{i}$ converges regularly to $\omega_{1,\infty}$
locally in $\R^n$.
 Therefore, we can assume that $\lim_{i \to \infty}\frac {r_{i}}{r_{i,1}}=\infty $ up to a subsequence. Since there might be many different numbers $r_{i}\in
(\frac{Rr_{i,1}}{2},2\delta)$ satisfying (\ref{Fact}), we must classify these numbers.
  For any two numbers  $r_{i}$ and $\tilde r_{i}$ in $(\frac{R r_{i,1}}{2},2\delta)$ satisfying (\ref{Fact}), they can be classified in different classes by the following
properties:
 \begin{align}\label{group1}
 \lim_{i \to \infty} \frac {r_{i}}{\tilde r_{i}}=+\infty &\quad\mbox{or } \quad \lim_{i \to \infty} \frac {r_{i}}{\tilde r_{i}}=0;
  \\ \label{group2}
  0< \liminf_{i \to \infty} \frac {r_{i}}{\tilde r_{i}}\leq \limsup_{i \to \infty} \frac {r_{i}}{\tilde r_{i}}<\infty .
  \end{align}
 We say that $\{r_{i}\}$ and $\{\tilde r_{i}\}$ are in the same class if they satisfy (\ref{group2}). Otherwise, they are in different classes if they satisfy (\ref{group1}).

 It can be seen that the number of above  different classes of $\{r_{i}\}$ satisfying (\ref{Fact}) must be finite.
Let  $\{\tilde r_{i}\}$ be any number satisfying  (\ref{Fact}) in the same class of $\{r_{i}\}$. Then there is an uniform positive integer $N_1$ such that
\begin{equation}\label{class}
\frac 1{N_1}\leq  \liminf_{i \to \infty} \frac {r_{i}}{\tilde r_{i}}\leq \limsup_{i \to \infty} \frac {r_{i}}{\tilde r_{i}}\leq N_1.
\end{equation}
 Otherwise, it will contradict  with the fact that  there is no infinitely number of above   $r_{i}\in (0,\delta_0)$ with disjoint annuluses ${B_{2r_i}\backslash
B_{r_i}(x^{j})}$ satisfying (\ref{Fact}).  Therefore, these numbers $\tilde r_{i}$ can be ruled out by letting $\delta_0$ sufficiently small and $R_0$ sufficiently large.

  We say that the class of $\{r_{i}\}$ is smaller than the class of $\{\tilde r_{i}\}$ if  $\lim_{i\to \infty} \frac {r_{i}}{\tilde r_{i}}=0 $, so we  can give an order for
such equivalent classes by $\{r_{2,i}\}\leq \{r_{3,i}\}\leq \cdots \leq \{r_{L,i}\}$ for some positive integer $L>0$ depending only on the energy bound $K$ and $\varepsilon$.
Then
we can separate the neck region $B_\delta \setminus B_{Rr_{i,1}}(x^j)$ by the following finite sum:

\begin{align*}\label{i}
 &\quad E_n(u_i; B_\delta \setminus B_{Rr_{1,i}}(x^j))\\
 &= E_n(u_i; B_\delta \setminus B_{Rr_{L,i}}(x^j))  + E_n(u_i;  B_{Rr_{L,i}}(x^j) \setminus B_{ \delta r_{L,i}}(x^j))\\
 &\quad+  E_n(u_i;  B_{ \delta r_{L,i}}(x^j) \setminus B_{Rr_{L-1,i}}(x^j))+ E_n(u_i;  B_{Rr_{L-1,i}}(x^j) \setminus B_{ \delta r_{L-1,i}}(x^j))+\cdots\\
 &\quad     +   E_n(u_i; B_{R r_{2,i}}(x^j) \setminus B_{ \delta r_{2,i}}(x^j))+   E_n(u_i;  B_{ \delta r_{2,i}}(x^j) \setminus B_{Rr_{1,i}}(x^j)).
\end{align*}

For a  sequence  $\{r_{2,i}\}$ in the smallest  class  satisfying  (\ref{Fact}) with the fact that
$\lim_{i \to \infty} \frac {r_{2,i}}{ r_{1,i}}=\infty $ and $\lim_{i\to \infty} r_{2,i}=0$,
 set
\[\tilde u_{2, i}(\tilde x) =u_{i} ( x^j+ r_{2,i} \tilde x).\]
Then we note that
 \begin{align*}
 &&\lim_{R \to \infty} \lim_{\delta \to 0} \lim_{i\to \infty}   E_n(u_i;  B_{R r_{2,i}}(x^{j}) \setminus B_{\delta r_{2,i} }(x^j))\\
 &&=\lim_{R \to \infty} \lim_{\delta \to 0} \lim_{i\to \infty}   E_n(\tilde u_{2,i};  B_{R}(0) \setminus B_{\delta}(0)).
 \end{align*}

Passing to a subsequence,  $\tilde u_{2, i}$ converges  to a   $\omega_{2}$  locally in $B_{R }(0) \setminus B_{\delta}(0)$ away from a finite concentration set of $\{\tilde
u_{2, i}\}$. As $R\to \infty$ and $\delta\to 0$, $\omega_{2}$ is an $n$-harmonic map in $\R^n$ by removing singularities.
If $\omega_{2}$ is non-trivial on $\R^n$,  then $\omega_{2}$ is a new bubble, which is different from the bubble $\omega_{1}$.
The  above bubble connection  $\omega_{2}$ might be trivial. In this case, there is at least a concentration point $p\in B_2\backslash B_1$ of $\{\tilde u_{2, i}\}$ due to
(\ref{Fact}). At each concentration point $p$ of $\tilde u_{2, i}$, we can repeat the procedure in Step 1; i.e.  at each concentration point $p$ of $\tilde u_{2, i}$ in $B_{R
}(0) \setminus B_{\delta}(0)$ , there are sequences $x_{i}^p\to p$ and $\lambda^p_{i}\to 0$ such that
\[\tilde u_{2, i}(x_{i}^p+\lambda^p_{i}x )\to  \omega_{2,p},\]
where $\omega_{2,p}$ is a $n$-harmonic map  on $\R^n$.
Note that $\tilde u_{2,p,\infty}$ is also a bubble for the sequence $\{u_{i}(x^{j}+r_{i,2}x_{i}^p+r_{i,2}\lambda_{i}^p x)\}$.

Set $x^{2,p}_{i}=x_{j}+r_{i,2}x_{i}^p$. For each $p\in B_{R }(0) \setminus B_{\delta}(0) $, we have
\[
\frac {|x^{j}-x^{2,p}_{i}|}{r_{i}^1}= \frac {r_{i,2}} {r_{i,1}}|x_{i}^p| \to \infty  \mbox { as  }i\to \infty. \]
Therefore, the bubble $\omega_{2,p}$ at $p\neq 0$ is different from the bubble $\omega_{1}$.
We continue the above procedure for possible new multiple bubbles at each blow-up point $p$ again. Since there is a uniform bound $K$ for $E_n(u_i; M)$ and  each non-trivial
bubble on $S^n$  costs at least $\varepsilon_g$ of  the energy by  the gap theorem, the above process must stop after finite steps.

Furthermore, we note
\begin{align*}
 &\lim_{R \to \infty} \lim_{\delta \to 0} \lim_{i \to \infty} E_n(u_i;  B_{ \delta r_{2,i}}(x^j) \setminus B_{Rr_{1,i}}(x^j))\\
 &=\lim_{R \to \infty} \lim_{\delta \to 0} \lim_{i \to \infty}E_n(\tilde u_{2,i};  B_{\delta}(0) \setminus B_{\frac {Rr_{1,i}}{r_{2,i}}}(0)).
 \end{align*}
 Since  $\{r_{2, i}\}$ in the smallest  class  satisfying  (\ref{Fact}) with the fact that
$\lim_{i\to \infty} \frac {r_{1,i}}{r_{2, i}}=0$ and $\lim_{i\to \infty}r_{2,i}=0$, we can see that $u_{i}$ satisfies (\ref{small}) on $B_{r_{2,i}\delta }(0) \setminus
B_{Rr_{1,i} }(0)$. Otherwise,  there is a number $r_i\in  (\frac 12 Rr_{1,i}, 2\delta r_{2,i})$ satisfying  (\ref{Fact}), $r_i$ must be belong to the class of $\{r_{1,i}\}$ or
$\{r_{2, i}\}$.  In an equivalent class,  it can be ruled out by $R$ sufficiently large or letting $\delta$ sufficiently small.

Since $\lim_{i\to \infty}\frac {r_{i,1}} {r_{i,2}}=0$ and $\omega_{1}$ is  a bubble limiting map for the sequence  $\{u_{i} ( x^1_{i}+ r^1_{i}  x)=\tilde u_{2, i}(\frac
{r^1_{i}} {r^2_{i}}x)\}$,  then $p=0$ is also a concentration point of $\tilde u_{2, i}$ on $\R^n$.
Therefore
the bubble $\omega_{2,0}$  must be the same bubble $\omega_{1}$. Since the bubble  $\omega_{1}$  is produced by $u_i$ on $B_{Rr_{1,i}}(x^j)$, we separate it  from other
bubbles without repeating.

\medskip\noindent{\bf Step 3.} To find out all  multiple bubbles.

 Let  $r_{i,3}$ be in the second small class of numbers satisfying (\ref{Fact}) with $\lim_{i \to \infty} \frac {r_{i,3}}{r_{i,2}}=\infty$ and
$\lim_{i \to \infty}  r_{i,3}=0$.
Set
\[\tilde u_{3, i}(\tilde x) =u_{i} (x^{j} +r_{i,3}\tilde x).\]
 Passing to a subsequence,  $\tilde u_{3, i}$ converges   locally to a  $\omega_{3}$   away from a finite concentration set of $\{\tilde u_{3, i}\}$ on $\R^n\backslash
\{0\}$.
Then we can repeat the argument of Steps 1-2. All bubbles produced by $\tilde u_{3, i}$, except for those concentrated in $0$,  are different from Steps 1-2.
By induction, we can find out all bubbles in all cases of the finite different classes. Since there is at least one nontrivial bubble on each different classes, the total
number $L$ of equivalent classes depends only on $K$ and $\varepsilon_g$. By  the gap theorem of $n$-harmonic maps on $S^n$,  the above process must stop after finite steps.

  In summary,  at each class level $k$, the blow-up happens,  there are  finitely many blow-up points  and  bubbles on $\R^n$.
At each level $k$ and each bubble point $p_{k,l}$, there are sequences $\tilde x_{i}^{k,l}\to p_{k,l}$ and $r_{i,k}\to 0$ with $\lim_{i\to \infty}\frac {r_{i,k}}{r_{i,
k-1}}=\infty$ such that passing to a subsequence,
$\tilde u_{i, k,l}(x)=u_{i} (x^{k,l}_{i}+r_{i,k}  x)$ converges to  $\omega_{k,l}$, where $\omega_{k,l}$  is an $n$-harmonic map in   $\R^n$, where
$x^{k,l}_{i}=x^j+r_{i,k}\tilde x^{k,l}$.

In conclusion, there are  finite   numbers $r_{i,k}$, finite points $x^{k,l}_i$, positive  constants $R_{k,l}$, $\delta_{k,l}$ and a finite number  of non-trivial $n$-harmonic
maps
$\omega_{k,l}$ on $\R^n$ such that

\begin{align}\label{neckenergy}
& \lim_{t_i\to \infty} E_n(u_i; B_{R_0}(x_k)) \\
=&E_n(u_{T_1}; B_{R_0}(x_i) )+\sum_{k=1}^L\sum_{l=1}^{J_k}
E_n(\omega_{k,l}; \R^n)\nonumber\\
+ &\sum_{k=1}^L \sum_{l=1}^{J_k} \lim_{R_{k,l} \to \infty} \lim_{\delta_{k,l} \to 0} \lim_{i \to \infty} E_n(\tilde u_{k,l,i}; B_{\delta_{k,l}} \backslash B_{R_{k,l}r_{i,
k}}(x^{k,l}_i)).\nonumber
\end{align}
Moreover, at each neck region $B_{\delta_{k,l}} \backslash B_{R_{k,l}r^{k}_\alpha}(x^{k,l}_i)$ in (\ref{neckenergy}), for all $i$ sufficiently large, we have
 \begin{equation}\label{Basic} \int_{B_{2r} \backslash
B_r (x^{k,l}_{\a})}{|\nabla \tilde u_{k,l,i}|^n dV} \leq \varepsilon \end{equation}
 for all $r \in (\frac{R_{k,l}r^k_i}{4},2\delta_{k,l})$, where   $\varepsilon$ is a  fixed constant to be chosen sufficiently small.

\begin{proof}[Proof of Theorem \ref {Theorem 2}] Wei-Wang \cite {WW} proved an energy identity of a sequence
of regular approximated n-harmonic maps $u_i$ in $W^{1,n}(M,N)\cap C^0(M,N)$, whose tension fields $h_i$
are bounded in $L^{n/n-1}(M)$. Let $u_i(x)=u(x,t_i)$ satisfy the equation (\ref{NA}). In this case,
$h_i:=(1+a|\nabla u_i|^{n-2})\partial_t u_i$, which is bounded in $L^{n/n-1}(M)$. In fact,
 using H\"older's inequality, we have
 \begin{equation*}\int_M \left (|\nabla  u_i|^{n-2} |\frac {\partial   u_i}{\partial  t}|\right )^{\frac n{n-1}} \leq
 \left ( \int_M  |\nabla  u_i|^{n}\right )^{\frac {n-2}{2(n-1)} } \left (\int  |\nabla   u_i|^{n-2} |\frac {\partial   u_i}{\partial   t}|^2\right )^{\frac n{2(n-1)}}\leq C.
 \end{equation*}
Under the condition (\ref {Basic}), we can apply Theorem B of \cite {WW} to prove
\begin{equation*}\lim_{R_{k,l} \to \infty} \lim_{\delta_{k,l} \to 0} \lim_{i \to \infty} E_n(\tilde u_{k,l,i}; B_{\delta_{k,l}} \backslash B_{R_{k,l}r_{i,
k}}(x^{k,l}_i))=0.\end{equation*}
 Therefore, the energy identity follows from (\ref{neckenergy}).
\end{proof}

\section{Minimizing the $n$-energy functional in homotopy classes}

In this section, we will present some applications of the related $n$-flow to minimizing the $n$-energy functional in a given homotopy class and give a proof of Theorem 3.
For a map $u:M\to N$, we recall the functional
\begin{equation}\label{4.1}
E_{n,\varepsilon}(u,M)=\int_Me_{n,\varepsilon}(u)\,dv,\end{equation}
where we set   $e_{n,\varepsilon}(u)=\frac {\varepsilon} 2 |\nabla u|^2  +\frac {1}{n} |\nabla u|^{n}+\frac{\varepsilon} {n+1}|\nabla u|^{n+1}$.

Let $u_i\in C^{\infty}(M,N)$ be a minimizing sequence of the $n$-energy in a homotopy class $[u_0]$. Since a minimizing sequence $u_i$ does not satisfy any equation, we cannot
have a good tool to use. Following an idea of the $\alpha$-harmonic map flow \cite {HY},
 we introduce a modified gradient flow for  the functional (\ref{4.1}) in the following:
 \begin{eqnarray}\label{34}
 &&(1-a+\varepsilon+a|\nabla u|^{n-2}+\varepsilon |\nabla u|^{n-1}) \frac {\partial u}{\partial t}\\
 &=&\frac{1}{\sqrt{|g|}}\frac{\partial}{\partial x_{i}}\left[ (\varepsilon + |\nabla
u|^{n-2}+\varepsilon |\nabla
u|^{n-1})g^{ij}\sqrt{|g|}\frac{\partial}{\partial x_{j}}u\right]\nonumber\\
&+&(\varepsilon +|\nabla
u|^{n-2}+\varepsilon |\nabla
u|^{n-1})A(u)(\nabla u,\nabla u)\nonumber
 \end{eqnarray}  with initial value $u(0)$ for a small constant $a>0$.
 Since the minimizing sequence $u_i$ is smooth, there is a   sequence $\varepsilon_i$ with  $\varepsilon_i \to 0$ such that
    \begin{equation}\label{eqn:alpha}
        \lim_{i\to \infty} E_{n,\varepsilon_i}(u_i, M)=\lim_{i\to \infty} E_n(u_{\varepsilon_i},M)=\inf_{u\in [u_0]}E_n(u,M).
    \end{equation}
Choosing  $u(0)=u_i$ to be  initial values, there is a sequence of $\varepsilon=\varepsilon_i\to 0$ such that the flow (\ref{34}) has a unique global smooth solution
$u_{\varepsilon_i}(x,t)$
on $M\times [0,\infty)$ with $u_{\varepsilon_i}(0)=u_i$.

By (\ref {34}), we have the energy identity
\begin{eqnarray} \label{47}
&& E_{n, {\varepsilon_i}} (u_{\varepsilon_i} (s),M)+
  \int_0^s \int_M  (1-a+\varepsilon+a |\nabla u_{\varepsilon_k}|^{n-2} +\varepsilon_i |\nabla
u|^{n-1})\left |\frac{\partial u_{\varepsilon_i}}{\partial t}\right |^2 dv\,dt\\
&&= E_{n, {\varepsilon_i}} (u_{i},M)\nonumber
\end{eqnarray}for each $s>0$.
  This implies that
    \begin{equation}\label{es1}
    \lim_{i\to \infty} \int_M\frac {\varepsilon_i} n| \nabla u_{\varepsilon_i}(s)|^{n+1} \,dv=0,
    \end{equation}
    \begin{equation}\label{es2}
   \lim_{i\to \infty} \int_0^s \int_M (1-a+\varepsilon +a |\nabla u_{\varepsilon_i}|^{n-2}+\varepsilon |\nabla
u_{\varepsilon_i}|^{n-1}) \abs{\partial_t u_{\varepsilon_i}}^2\, dv dt=0.
    \end{equation}
  Hence the sequence $\{u_{\varepsilon_i}(s)\}_{i=1}^{\infty}$  for each $s>0$ is also a minimizing sequence in the homotopic class $[u_0]$.

    \begin{lem}  \label{5.1}   Let $\rho$, $R$ be two constants with $\rho<R\leq 2\rho$. For any $x_0$ with $B_{2\rho}(x_0)\subset M$ and
for any two  $s,\tau \in [0, T)$, we have
\begin{align*}&\quad \int_{B_{\rho}(x_0) } e_{n,\varepsilon_i}(u_{\varepsilon_i}(\cdot ,s))\,dv-\int_{B_{R}(x_0)}  e_{n,\varepsilon_i}(u_{\varepsilon_i}) (\cdot , \tau
)\,dv\\
&\leq
C\int_{s}^{\tau }\int_M (1+a|\nabla u_{\varepsilon_i}|^{n-2}+\varepsilon_i |\nabla u|^{n-1})|\partial_t u_{\varepsilon_i}|^2\,dv\,dt\nonumber\\ &+C\left (\frac
{(\tau-s)}{(R-\rho)^2}\,\int_M e_{n,\varepsilon_i}(u_i)\,dv\, \int_{s}^{\tau
}\int_M (1+a|\nabla u_{\varepsilon_i}|^{n-2}+\varepsilon_i |\nabla u_{\varepsilon_i}|^{n-1})|\partial_t u_{\varepsilon_i}|^2\,dv\,dt \right )^{1/2}\,. \nonumber\\
\end{align*}
\end{lem}
\begin{proof}Let $\phi$ be a cut-off function in $B_R(x_0)$ such that $\phi=1$ in $B_{\rho}$ and $|\nabla \phi|\leq C/(R-\rho)$.  The required result follows from multiplying
(\ref{34}) by $\phi\partial_t u_{\varepsilon_i}$.\end{proof}

We can  repeat the same steps of Lemma \ref{boundness} to obtain
    \begin{lem} \label{5.3}  There exists a positive constant
    $\varepsilon_0<i(M)$ such that  if for some $R_0$ with $0<R_0<\min
\{\varepsilon_0, \frac {t_0^{1/2}}2\}$  the inequality
\[\sup_{t_0-4R_0^2\leq t<t_0}\int_{B_{2R_0}(x_0)}|\nabla u_{\varepsilon_i}|^{n} \,dv
   < \varepsilon_0\] holds,
we have
\[\|\nabla u_{\varepsilon_i}\|_{L^{\infty}(B_{R_0}(x_0))}\leq C(R_0) \]
where  $C$   is a constant independent of $\varepsilon$ and depends on $R_0$.
\end{lem}
Now we complete a proof of Theorem \ref {Theorem 4}.
\begin{proof}[Proof of Theorem \ref {Theorem 4}]
For a minimizing sequence $u_i$ of the $n$-energy in the homotopy class, let $u$ be the weak limit of $\{u_{i}\}_{i=1}^{\infty}$ in $W^{1,n}(M)$.
Set
 $$\Sigma_0 =\bigcap_{R>0} \left\{x_0\in \Omega : B_R(x_0)\subset  M,\quad \liminf_{i\to \infty}
  \int_{B_R(x_0)} |\nabla u_{i} |^n \,dx\geq \varepsilon_0\right
 \}
 $$
 for a small constant $\varepsilon_0>0$. It is known that $\Sigma_0$ is  a set of finite points.
 For the above sequence $\{u_{\varepsilon_i}(s)\}_{i=1}^{\infty}$, we set
 $$\Sigma_s =\bigcap_{R>0} \left\{x_0\in \Omega : B_R(x_0)\subset  M,\quad \liminf_{i\to \infty}
  \int_{B_R(x_0)} |\nabla u_{\varepsilon_i}(\cdot ,s) |^n \,dx\geq \varepsilon_0\right
 \},
 $$
 which is also finite.   Applying (\ref{es1})-(\ref{es2}) to Lemma \ref {5.1}, we obtain that $\Sigma_0=\Sigma_s$  for all $s>0$ (see a similar argument to one in \cite
{HTY}).
 By using Lemmas \ref{5.1}-\ref{5.3},   $|\nabla u_{\varepsilon_i}(x,s)|\leq C(R)$ on $P_R(x_0,s)$ for each $x_0\in M\backslash \Sigma$ with $B_R(x_0)\subset M$.
By this result, we know that $u(x,t)$ is a weak solution to the flow (\ref{34}).  Since $u_{i}(x,t)$
converges
    weakly to $u(x,t)$ in $W^{1,2}(M\times [0,1])$, $u(\cdot, t)\equiv u(\cdot,0)=u$. Then $u(x,t)$
      is an $n$-harmonic map from $M$ to $N$ independent of $t\in [0,1]$. By the regularity result on $n$-harmonic maps,
    $u$ is a smooth map on $M$.

For any $x_0\in M\backslash \Sigma$, there is a constant $R>0$ such that  $B_{R}(x_0)\subset M\backslash \Sigma$.
 Note that $u_{\varepsilon_i}(\tau)$ converges strongly to $u$ in  $W^{1,n}(B_{R}(x_0))$.
As $i\to\infty$, we apply Lemma 5.1 to obtain that
\begin{align*} & \int_{B_{\rho}(x_0)}\frac 1 n |\nabla u|^n \leq \liminf_{i\to\infty}\int_{B_{\rho}(x_0) }  \frac 1 n |\nabla u_i|^n \,dv\leq
\limsup_{i\to\infty}\int_{B_{\rho}(x_0) } e_{\varepsilon_i}(u_i)\,dv\\
&\leq \limsup_{i\to\infty}   \int_{B_{R}(x_0)}  e_{n,\varepsilon_i}(u_{\varepsilon_i}) (\cdot , \tau )\,dv=
 \int_{B_{R}(x_0)} \frac 1 n |\nabla u|^n  \,dv\end{align*}
for any $R$ with $\rho<R$. Letting $R\to \rho$, we have
\[\int_{B_{\rho}(x_0)}\frac 1 n |\nabla u|^n = \lim_{i\to\infty}\int_{B_{\rho}(x_0) } \frac 1 n |\nabla u_i|^n \,dv.\]
This implies that $u_i$ converges strongly to $u$ in $W^{1,n}(B_{\rho}(x_0))$ and hence strongly in  $W^{1,n}_{loc}(M\backslash\Sigma)$.

Next, we use  a similar proof of Sacks-Uhlenbeck \cite{SU} to show that $\Sigma_0=\Sigma_s=\emptyset $ if $\pi_n(N)=0$. Let $\{u_{\varepsilon_i}(s)\}_{i=1}^{\infty}$ be the
above sequence.
it is known that $u_{\varepsilon_i}(s)$ converges to $u$ strongly in $W^{1, n+1}_{loc}(M\backslash \Sigma_s)$.
Without loss of generality, we assume that there is one singularity $x^1$ in $\Sigma_s$.
    Let $\eta (r)$ be a smooth cutoff function in $\Bbb R$ with the property that $\eta\equiv 1$ for $r\geq 1$ and $\eta\equiv
    0$ for $r\leq 1/2$. For some $\rho>0$, we define a new sequence of maps
     ${v}_i: M\to N$ such that ${v}_i$ is the same as $u_i$ outside $B_\rho(x_1)$, and for  $x\in
     B_\rho(x_1)$,
    \begin{equation*}
        {v}_i(x)=\exp_{u(x)} \left( \eta(\frac{\abs{x}}{\rho})\exp^{-1}_{u(x)}\circ u_{\varepsilon_i}(x,s) \right),
    \end{equation*}
    where $\exp$ is the exponential map on $N$.
Note that ${v}_i\equiv u$ on $B_{\rho/2}(x_1)$
    and ${v}_i\equiv u_{\varepsilon_i}(s)$
    outside $B_\rho(x_1)$ and that   $u_{\varepsilon_i}(s)$ converges to $u$ on $B_{\rho}(x_1)\setminus B_{\rho/2}(x_1)$
    strongly in $W^{1,n+1}$ and thus in $C^\beta$ for some $\beta>0$. Hence for sufficiently large $i$, $v_i(B_{\rho(x_1)}\setminus B_{\rho/2}(x_1))$ lies in a small
neighborhood of $u(x_1)$,
     where $\exp^{-1}_{u(x)}$ is a well defined smooth map (if $\rho$ is small). Since $F(y)=\exp_{u(x)}\left( \eta(\frac{\abs{x}}{\rho})\exp^{-1}_{u(x)} y \right)$ is
     a smooth map from a neighborhood of $u(x_1)$ into itself, we have
    \begin{eqnarray*}
        \int_{B_\rho\setminus B_{\rho/2} (x_1)}|\nabla (v_i-u)|^n\,dv   &=&
         \int_{B_\rho\setminus B_{\rho/2} (x_1)} |\nabla (F\circ u_{\varepsilon_i}(s)- F\circ u)|^n  \,dv \\
        &\leq& C  \int_{B_\rho\setminus B_{\rho/2} (x_1)}  |\nabla (u_{\varepsilon_i}(s)-u)|^n  \,dv \to 0
    \end{eqnarray*}  as $i\to\infty$.
   It implies that
    \begin{equation}\label{eqn:modify}
        \norm{{v}_i-u}_{W^{1,n}(M)}\to 0
    \end{equation}
    as $i\to \infty$.

    Since $\pi_n(N)$ is trivial, ${v}_i$ is in the same homotopy class as $u_{\varepsilon_i}(s)$. Since $u_{\varepsilon_i}(s)$ is a minimizing sequence of $E_{n,
\varepsilon_i}$ and
    $u_{\varepsilon_i}(s)$ converges weakly to $u$ in $W^{1,n}$,  we have
    \begin{align*}
    \label{eqn:weak}
       &E_n(u)\leq \limsup_{i\to \infty}E_{n, \varepsilon_i}(u_{\varepsilon_i}(s))
       \leq \limsup_{i\to \infty}E_{n, \varepsilon_i}(v_{i})=E_n(u),\end{align*}
   which implies that $u_{\varepsilon_i}(s)$ converges to $u$ strongly in $W^{1,n}(M, N)$, which means that there is no energy concentration; i.e. $\Sigma_0=\Sigma_s=\emptyset
$.

    \end{proof}

 \section{Minimizing the $p$-energy functional in homotopy classes}

For a small $\varepsilon >0$,  we introduce a perturbation  of  the $p$-energy functional by
\begin{equation} \label{p-pert}
    E_{p,\varepsilon}(u; M)=\int_{M} \frac {\varepsilon} 2 \abs{\nabla u}^2 +\frac 1p \abs{\nabla u}^p +\frac  \varepsilon {n+1}|\nabla u|^{n+1} dv.
\end{equation}
The  Euler-Lagrange equation associated to this functional is
 \begin{align}
 \label{p-eq}
  &\nabla \cdot ([\varepsilon +|\nabla u|^{p-2} +\varepsilon|\nabla
u|^{n-1}]\nabla u) \\
&+[\varepsilon +|\nabla u|^{p-2} +\varepsilon|\nabla
u|^{n-1}]A(u)(\nabla u,\nabla u)=0.\nonumber
    \end{align}
The  gradient flow for  the above equation  is
 \begin{eqnarray}\label{p-AF}
\frac {\partial u}{\partial t}&=&\mbox{dvi} \left[ (\varepsilon +|\nabla u|^{p-2} + \varepsilon |\nabla
u|^{n-1})\nabla u\right]\\
&&+(\varepsilon + |\nabla u|^{p-2} +\varepsilon|\nabla
u|^{n-1})A(u)(\nabla u,\nabla u)\nonumber
    \end{eqnarray}
    with initial value $u(0)=u_0$ in $M$. If the initial map $u_0$ is smooth, there is a global smooth solution to (\ref{p-AF}) by using  proofs in \cite {Hung} and \cite
{FR}.

    Without loss of generality, we assume $g_{ij}=\delta_{ij}$. Then we have

\begin{lem} \label{Lemma2}  Let $u$ be a solution of the equation (\ref {p-AF}). Then for all $\rho \leq R$ with $B_R(x_0)\subset
\Omega$, we have
\begin{align*} &\rho^{p-n}\int_{B_{
\rho}(x_0)} \left [ \frac {\varepsilon} 2 |\nabla u|^2+\frac 1 p |\nabla u |^p +\frac 1 {n+1}\varepsilon |\nabla
u|^{n+1}\right ]\,dx\\
& +\frac {n(p-2)}{2p}\int_{\rho}^R r^{p-1-n}\int_{B_{r}} {\varepsilon}  |\nabla u|^2\,dx\,dr \\
&+\int_{B_R\backslash B_{\rho}(x_0)}\left [\frac 1 2
|\partial_r u|^2+ \frac 1 {n+1}\varepsilon |\nabla u|^{n-1} |\partial_r
u|^2\right ] r^{p-n} \,dx\\
= \,& R^{p-n}\int_{B_R(x_0)} \left [ \frac {\varepsilon} 2 |\nabla u|^2+\frac 1 p|\nabla u
|^p +\frac 1 {n+1}\varepsilon |\nabla u|^{n+1}\right ]\,dx\\
& +\frac {n+1-p} {n+1}\int_{\rho}^R \int_{B_r(x_0)}r^{p-1-n}
\varepsilon |\nabla u|^{n+1}\,dx  \,dr\\
&+\int_{\rho}^R \int_{B_{r}}r^{p-1-n}\left <\frac{\partial u}{\partial t},\,x_i  \nabla_i
u\right >\,dx\,dr.  \end{align*}
\end{lem}
 \begin {proof} Without loss of generality, we assume that $x_0=0$. Multiplying  (\ref {p-AF}) by $x_i \nabla_i u$, we have
 \[
\int_{B_{r}}\left <\frac{\partial u}{\partial t},\,x_i  \nabla_i
u\right >-\left <div\left(  (\varepsilon +|\nabla u|^{p-2} + \varepsilon |\nabla
u|^{n-1})\nabla u\right),\, x_i
\nabla_i u\right >\,dx=0.\]
Integration by parts yields that
\begin{align*}
&\quad \int_{B_{r}}\left <\frac{\partial u}{\partial t},\,x_i  \nabla_i
u\right >\,dx-\frac 1 r\int_{\partial B_{r}}  (\varepsilon +|\nabla u|^{p-2} + \varepsilon |\nabla
u|^{n-1})|x_i \nabla_i u|^2  \,d\omega
\\
&=- \int_{B_{r}} (\varepsilon |\nabla u|^2+|\nabla u|^{p} + \varepsilon |\nabla
u|^{n+1})  +  \frac 12 (\varepsilon +|\nabla u|^{p-2} + \varepsilon |\nabla
u|^{n-1})
x_i  \nabla_{i} (|\nabla u|^2)    \,dx\nonumber\\
 &=  \int_{B_{r}} \left (\frac {(n-2)\varepsilon} 2 |\nabla u|^2+\frac {(n-p)} p|\nabla u|^{p} - \frac {\varepsilon}{n+1} |\nabla
u|^{n+1}\right )  \,dx\nonumber\\
&\quad -r\int_{\partial B_{r}}\left (\frac {\varepsilon} 2 |\nabla u|^2+\frac {1} p|\nabla u|^{p} + \frac {\varepsilon}{n+1} |\nabla
u|^{n+1}\right ) \,d\omega .
\end{align*}
Multiplying by $r^{p-1-n}$ both sides of the above identity, we have
\begin{align*}
&\quad \frac {d}{dr} \left [r^{p-n}\int_{B_{r}}\left (\frac {\varepsilon} 2 |\nabla u|^2+\frac {1} p|\nabla u|^{p} + \frac {\varepsilon}{n+1} |\nabla
u|^{n+1}\right )dx \,\right ]\\
&\quad -\frac {n(p-2)}{2p}r^{p-1-n}\int_{B_{r}} {\varepsilon}  |\nabla u|^2\,dx +r^{p-1-n}\int_{B_{r}} \frac {\varepsilon (n+1-p)}{n+1} |\nabla
u|^{n+1}dx
\\
&=- r^{p-1-n}\int_{B_{r}} \left (\frac {(n-2)\varepsilon} 2 |\nabla u|^2+\frac {(n-p)} p|\nabla u|^{p} - \frac {\varepsilon}{n+1} |\nabla
u|^{n+1}\right )  \,dx\nonumber\\
&\quad +r^{p-n}\int_{\partial B_{r}}\left (\frac {\varepsilon} 2 |\nabla u|^2+\frac {1} p|\nabla u|^{p} + \frac {\varepsilon}{n+1} |\nabla
u|^{n+1}\right ) \,d\omega \\
&= r^{p-n}\int_{\partial B_{r}}  (\varepsilon +|\nabla u|^{p-2} + \varepsilon |\nabla
u|^{n-1})| \partial_r u|^2  \,d\omega -\int_{B_{r}}r^{p-1-n}\left <\frac{\partial u}{\partial t},\,x_i  \nabla_i
u\right >\,dx.
\end{align*}
Then integrating with respect to $r$ from $\rho$ to $R$ yields  the
result.
\end{proof}

\begin{lem}\label{lemma 6.2}  Let $u_i\in C^{\infty}(M,N)$ be a minimizing sequence in the homotopy class $[u_0]$.
 Then there are a sequence of $\varepsilon_i\to 0$ and solutions $u_{\varepsilon_i}$ of equation (\ref{p-AF}) with initial value $u_i$ such that
 $u_{\varepsilon_i}( t)$ for all $t\in [0,\infty)$ is also a minimizing sequence in the same homotopy class.
 Moreover, there is a uniform $\tilde t\in [1/2,1]$ such that
 \begin{equation*}\lim_{i\to \infty}\int_M\frac  {\varepsilon_i} {n+1}|\nabla u_{\varepsilon_i}|^{n+1}(\cdot,\tilde t ) \,dv+\lim_{i\to \infty} \int_M \abs{\partial_t
u_{\varepsilon_i}(\cdot, \tilde t)}^2 dv =0.\end{equation*}
\end{lem}
\begin{proof} Since the minimizing sequence $u_i$ is smooth, there is a   sequence  $\varepsilon_i \to 0$ such that
    \begin{equation*}
        E_{p,\varepsilon_i}(u_i)\leq E_p(u_{\varepsilon_i} )+\frac{1}{i},
    \end{equation*}
which implies
    \begin{equation}\label{eqn:alpha}
        \lim_{i\to \infty} E_{p, \varepsilon_i}(u_i)=\lim_{i\to \infty} E_p(u_{\varepsilon_i})=\inf_{v\in [u_0]}E_p(v).
    \end{equation}
   Then  there is a unique solution $u_{\varepsilon_i}(x,t)$ to the flow (\ref {p-AF}) with initial value $u_{\varepsilon_i}(0)=u_i$. Similar to   Lemma \ref {EI}, we have
    \begin{equation*}
       E_{\varepsilon_i}(u_{\varepsilon_i}(\cdot,\tau))+\int_0^\tau \int_M \abs{\partial_t u_{\varepsilon_i}}^2 \, dv dt= E_{\varepsilon_i}(u_{\varepsilon_i}).
    \end{equation*}
This implies that  $u_{\varepsilon_i}(x,\tau )$ for $\tau$ is a minimizing sequence of $E$ in the homotopy class  $[u_0]$,
which yields
    \begin{equation}\label{eqn:harmonic}
      \lim_{i\to \infty}\int_M\frac  {\varepsilon_i} {n+1}|\nabla u_{\varepsilon_i}|^{n+1}(\cdot,\tau ) \,dv+ \int_0^\tau \int_M \abs{\partial_t u_{\varepsilon_i}}^2 dv dt=0.
    \end{equation}
Then there is a uniform $\tilde t\in [1/2,1]$ such that
\begin{equation*}\lim_{i\to \infty} \int_M \abs{\partial_t u_{\varepsilon_i}(\cdot, \tilde t)}^2 dv =0.
\end{equation*}
\end{proof}

\begin{proof}[Proof of Theorem \ref {Theorem 5}]
Let $u_i\in C^{\infty}(M,N)$ be a minimizing sequence in a homotopy class and $u$ its weak limit.
If $N$ is  a homogeneous manifold, we claim that $u$ is a
weak $p$-harmonic map from $M$ and $N$.

Let $X=(X^1,\cdots, X^L)$ be a Killing vector on $N \subset \R^L$ as in H\'elein \cite{Helein} and $u=(u^1,\cdots, u^L)\in N$.
Let $\varphi$ be a cut-off function compactly
supported in $M$. Since $u_\varepsilon$ is a solution of (\ref {p-AF}), we use $\varphi X(u)$ as a testing vector
to get
\begin{equation*}
    \int_{M} \langle \nabla_k(\varphi X(u_\varepsilon)),
      (\varepsilon +\abs{\nabla u_\varepsilon}^{p-2} + \varepsilon \abs{\nabla u_\varepsilon}^{n-1}) \nabla_k u_\varepsilon\rangle dv=-\int_M \langle\varphi X(u_\varepsilon),
\partial_t  u_{\varepsilon}\rangle .
\end{equation*}
Since $X$ is a Killing vector, it implies that $\sum_{l,m=1}^L \nabla_k u_{\varepsilon}^m \nabla_m X^l \nabla_k u_{\varepsilon}^l =0$, so
\begin{equation}\label{eqn:killing}
    \int_{M} \langle \nabla_k \varphi \, X(u_\varepsilon), (\varepsilon +\abs{\nabla u_\varepsilon}^{p-2} + \varepsilon \abs{\nabla u_\varepsilon}^{n-1})
     \nabla_k u_\varepsilon\rangle dv=-\int_M \langle\varphi X(u_\varepsilon), \partial_t u_{\varepsilon}\rangle.
\end{equation}

 Let $u$ be the weak limit of $u_{\varepsilon_i}$ in $W^{1,p}(M\times [0,1])$ by passing  to a subsequence if necessary.
    By  a compact result in \cite {CHH}, $|\nabla u_{\varepsilon_i}|^{p-2}\nabla u_{\varepsilon_i}$ converges weakly to $|\nabla u|^{p-2}\nabla u$ in $L^{p^*}$ with $\frac 1 p
+\frac 1 {p^*}=1$.
     Since $u_{\varepsilon_i}$ converges to $u$ strongly in $L^p$ and
$X$ is a smooth vector   on $N$, $X(u_{\varepsilon_i})$
converges to $X(u)$ strongly in $L^p$. Letting $\varepsilon_i$
go to zero in equation (\ref{eqn:killing}) and noting (\ref{eqn:harmonic}), we
have
\begin{equation*}
  \int_{M} \nabla_k \varphi \langle X(u), |\nabla u|^{p-2}\nabla_k u\rangle d\mu=0,
\end{equation*}
which implies
\begin{equation*}
  \int_{M} \langle \nabla_k (\varphi X(u)), |\nabla u|^{p-2}\nabla_k u\rangle d\mu=0
\end{equation*} due to the fact that $X$ is a Killing vector.
Since $N$ is a homogeneous space, we apply the construction of a Killing field $\{X_j\}$
by Helein \cite{Helein} and  choose $\varphi_j$ to obtain that
\begin{equation*}
  \sum_{j}\varphi_j X_j(u)
\end{equation*}
is any compactly supported vector field (along $u$).
This implies
that $u$ is a weak $p$-harmonic map. We know that $u$ is a weak solution to the $p$-harmonic map flow. It follows from (\ref{eqn:harmonic})
     that $u$ is a map independent of $t\in [0,1]$. Since $u_{\varepsilon_i}(x,t)$
converges
    weakly to $u(x,t)$ in $W^{1,2}(M\times [0,1])$. Hence, $u(\cdot, t)\equiv u(\cdot,0)$
    is a (weakly) $p$-harmonic map from $M$ to $N$.

We define
 $$\Sigma =\bigcap_{R>0} \left\{x_0\in \Omega : B_R(x_0)\subset  M,\quad \liminf_{\varepsilon_i\to 0}
  \frac 1 {R^{n-p}}\int_{B_R(x_0)} |\nabla u_{\varepsilon_i} |^p \,dx\geq \varepsilon_0\right
 \}
 $$
 for a sufficiently small constant $\varepsilon_0$. Then, $\mathcal
 H^{n-p}(\Sigma )<+\infty$.
For any $x_0\notin \Sigma$ with $B_{R_0}(x_0)\subset M \backslash \Sigma$,
for each $y\in B_{R_0/2}(x_0)$ and for each $\rho\in (0,R_0/2)$,
we have
 \begin{align}\label{mono}
 & \rho^{p-n} \int_{B_{\rho }(y)}|\nabla u|^p\,d  M\leq
 \lim_{\varepsilon_i\to 0}
\rho^{p-n}\int_{B_{ \rho}(y)} |\nabla u_{\varepsilon_i} |^p
\,d M <\varepsilon_0
  \end{align}
 for a sufficiently small constant $\varepsilon_0>0$.
 Since $u$ is a weakly $p$-harmonic map satisfying (\ref{mono}),  it follows from a similar proof of stationary $p$-harmonic maps into homogenous manifolds    (see  \cite
{TW})
 that $u$ belongs to $C^{1,\a}_{loc}(M\backslash \Sigma)$.
\end{proof}

\begin{acknowledgement}
 {The research  of the
 author was supported by the Australian Research Council
grant DP150101275.}
\end{acknowledgement}

\end{document}